\documentclass[11pt,a4paper]{article}

\usepackage{amsmath}
\usepackage{amsfonts}
\usepackage{amsthm}
\usepackage{amssymb}
\usepackage{enumitem}
\usepackage{graphicx}
\usepackage{mathrsfs}
\usepackage{xcolor}
\usepackage{hyperref}
\usepackage{fullpage}
\usepackage{array}
\usepackage{adjustbox}

\usepackage{pgfplots}
\usetikzlibrary{math,calc,arrows,shapes.misc,positioning}
\pgfplotsset{compat=1.16}

\usepackage{xpatch}
\tracingpatches
\xpatchcmd{\proof}{\itshape}{\bf}{}{}
\xpatchcmd{\example}{\itshape}{\bf}{}{}

\usepackage{todonotes}

\numberwithin{equation}{section}

\newtheorem{thm}{Theorem}[section]
\newtheorem{lemma}[thm]{Lemma}
\newtheorem{prop}[thm]{Proposition}
\newtheorem{cor}[thm]{Corollary}

\theoremstyle{definition}
\newtheorem{rmk}[thm]{Remark}

\theoremstyle{remark}
\newtheorem{example}[thm]{Example}

\newcommand{\R}{{\mathbb R}}
\newcommand{\C}{{\mathbb C}}
\newcommand{\Z}{{\mathbb Z}}

\DeclareMathOperator{\E}{\mathbb{E}}
\let\P\undefined
\DeclareMathOperator{\P}{\mathbb{P}}

\newcommand{\Var}{\mathrm{Var}}

\newcommand{\cP}{{\mathcal P}}

\newcommand{\CC}{\mathcal{C}}

\newcommand{\hA}{{\widehat A}}

\newcommand{\hH}{{\widehat H}}

\newcommand{\hS}{{\widehat S}}

\newcommand{\hOmega}{{\widehat \Omega}}

\newcommand{\hX}{{\widehat X}}

\newcommand{\hbeta}{{\hat \beta}}
\newcommand{\hmu}{{\hat \mu}}
\newcommand{\htau}{{\hat \tau}}
\newcommand{\hpi}{{\hat \pi}}

\newcommand{\eps}{\varepsilon}

\renewcommand{\Re}{\operatorname{Re}}

\newcommand{\TV}{\mathrm{TV}}

\newcommand{\floor}[1]{\lfloor #1 \rfloor}

\newcommand{\bone}{\mathbf{1}}
\newcommand{\mbi}{\mathbf{i}}

 \usepackage{stmaryrd}

\usepackage{mathrsfs}

\newcommand*\circled[1]{\tikz[baseline=(char.base)]{
    \node[shape=circle,draw,inner sep=0.15em] (char) {$\mathrm{#1}$};}}

\title{Rates of memory loss for null recurrent Markov chains}

\author{Ilya Chevyrev\thanks{
Institut für Mathematik, TU Berlin, 10623 Berlin, Germany
and
School of Mathematics, The University of Edinburgh, Edinburgh, EH9 3FD, UK.
\href{mailto:ichevyrev@gmail.com}{\tt ichevyrev@gmail.com}}
\and Alexey Korepanov\thanks{\href{mailto:khumarahn@gmail.com}{\tt khumarahn@gmail.com}}}

\date{17 January 2025}

\begin{document}

\maketitle

\begin{abstract}
    Orey (1962) proved that for an irreducible, aperiodic, and recurrent Markov chain with transition operator $P$, the sequence $P^n (\mu - \nu)$ converges to zero in total variation for any two probability measures $\mu$ and $\nu$. In other words, all such Markov chains exhibit \emph{memory loss}. While the \emph{rates} of memory loss have been extensively studied for positive recurrent chains, there is a surprising lack of results for null recurrent chains. In this work, we prove the first estimates of memory loss rates in the null recurrent case.
\end{abstract}
%

\renewcommand{\baselinestretch}{0.0}\normalsize
\setcounter{tocdepth}{1}
\tableofcontents
\renewcommand{\baselinestretch}{1.0}\normalsize

\section{Introduction}
\label{sec:intro}

Consider a Markov chain $X_0,X_1,\ldots$ taking values in a measurable space $\Omega$ with transition operator $P\colon \cP(\Omega)\to\cP(\Omega)$,
where $\cP(\Omega)$ is the space of probability measures on $\Omega$.
That is, if $X_n \sim \mu$ with $\mu \in \cP(\Omega)$, then $X_{n+1} \sim P \mu$.
A question of fundamental interest is whether the chain is \emph{mixing}, and if so, what is the \emph{mixing rate}.
A well studied notion of mixing is \emph{memory loss:}
the convergence $\|P^n\mu - P^n\nu\|_{\TV} \to 0$ as $n\to\infty$ for $\mu,\nu \in \cP(\Omega)$ that act as initial distributions for $X$,
where $\|\cdot\|_{\TV}$ is the total variation norm.

A celebrated result of Orey~\cite{Orey_59_recurrent,Orey_62_ergodic} (see also~\cite{Jamison_Orey_67_MC,Orey_71_MC}) states that if $X$ is recurrent, irreducible and aperiodic, then $\lim_{n\to\infty}\|P^n\mu - P^n\nu\|_{\TV} = 0$ for all $\mu,\nu \in \cP(\Omega)$.
This result is remarkable in its scope, but of course, at this level of generality, comes with no rate of convergence.

In the case that $X$ is \emph{positive recurrent}, and thus admits an invariant probability measure,
there are well-known results that quantify memory loss,
e.g.\ for so-called \emph{regular} measures, one has $\sum_{n=1}^\infty \|P^n\mu - P^n\nu\|_{\TV} < \infty$, so in particular $\|P^n\mu - P^n\nu\|_{\TV} \lesssim n^{-1}$, see~\cite[Theorem~6.11]{Nummelin_84_MC}.
There is an abundance of criteria for polynomial and geometric rates of convergence, e.g.\ based on estimates of first return times to nice sets, or the existence of Lyapunov functions, see
\cite{Nummelin_Tweedie_78_geom_erg, Nummelin_Tuominen_82_geom_erg, Nummelin_Tuominen_83_rates}
for early works
as well as the books
\cite{Lindvall_02_coupling,Nummelin_84_MC,Benaim_Hurth_2022_MC_Metric,Meyn_Tweedie_09_MC} and the references therein.
For results in the case of subgeometric mixing rates, see for example~\cite{Tuominen_Tweedie_94_rates,Veretennikov_97_mixing,Jarner_Roberts_02_rates,
Douc_et_al_04_subgeo,
Baxendale_05_renewal,
Hairer16_MP,CM23_rates}
and the book~\cite{DMPS18_MC}.
To our knowledge, all of these conditions are restricted to the \emph{positive recurrent} case.

The purpose of this paper is to provide the first rates in convergence $\|P^n\mu - P^n\nu\|_{\TV} \to 0$
for \emph{null recurrent} Markov chains.
We present three results which use different strategies and work under different assumptions, and which are applicable to general Harris chains.

To illustrate the results, consider an aperiodic and recurrent Markov chain $X_n$
on $\Omega = \{0,1,2,\ldots\}$ with connection graph
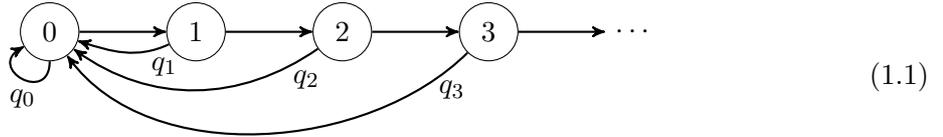
\begin{equation}
    \label{eq:RMC}
    \begin{aligned}
        \adjustbox{trim=0 1em 0 0.0em,clip=true}{ 
            \begin{tikzpicture}[node distance=5em,minimum size=2em]
                \node[circle,draw] (w1) {0};
                \node[circle,draw] (w2) [right of=w1] {1};
                \node[circle,draw] (w3) [right of=w2] {2};
                \node[circle,draw] (w4) [right of=w3] {3};
                \node[]            (w5) [right of=w4] {$\cdots$};
                \path[->,>=stealth',thick]
                (w2) edge[<-] (w1)
                (w3) edge[<-] (w2)
                (w4) edge[<-] (w3)
                (w5) edge[<-] (w4)
                (w1) edge [<-,out=-150,in=-90,looseness=3.8] node[below,minimum size=1.25em] {$q_0$} (w1)
                (w1) edge [<-,out=-18,in=-155,looseness=1] node[below,minimum size=1.1em,pos=0.95] {$q_1$} (w2)
                (w1) edge [<-,out=-36,in=-145,looseness=1] node[below,minimum size=1.5em,pos=0.95] {$q_2$} (w3)
                (w1) edge [<-,out=-54,in=-135,looseness=0.9] node[below,minimum size=1.5em,pos=0.95] {$q_3$} (w4)
                ;
            \end{tikzpicture}
        }
    \end{aligned}
\end{equation}
and transition probabilities $\P(X_1 = 0 \mid X_0 = n) = q_n$, $\P(X_1 = n+1 \mid X_0 = n) = 1 - q_n$.
Consider initial distributions $\mu = \delta_0$ and $\nu = P^\ell \delta_0$,
where $\delta_n$ gives probability 1 to the state~$\circled{n}$.

\begin{rmk}\label{rem:motivate_model}
The Markov chain~\eqref{eq:RMC} and our choice of initial distributions are far more than a convenient illustration.
As we show in Section~\ref{sec:general},
estimates on $\| P^{n+\ell} \delta_0 - P^n \delta_0 \|_{\TV}$ for~\eqref{eq:RMC}
lead directly to memory loss rates for Markov chains on a general state space with relaxed recurrence
properties, specifically for Harris chains.
\end{rmk}

For $n\geq1$ let
\begin{align*}
    p_n
    = \P(\inf \{k \geq 1 \,:\, X_k = 0\} = n\mid X_0 = 0)
    = (1 - q_0) \cdots (1 - q_{n-2}) q_{n-1}
    \;.
\end{align*}
Suppose that
\[
\sum_{k > n} p_k \leq C n^{-\alpha}
\]
with some $C, \alpha > 0$. If $\alpha > 1$, then $X_n$ is positive recurrent,
and it is generally null recurrent otherwise.

A standard result (see Appendix~\ref{sec:Lind}) for positive recurrent Markov chain is that for $\alpha > 1$,
\begin{equation}
    \label{eq:best}
    \| P^{n+\ell} \delta_0 - P^n \delta_0 \|_{\TV}
    \lesssim \frac{\ell}{n^\alpha}
    \;,
\end{equation}
where $\lesssim$ means a bound up to a multiplicative constant that does not depend on $n$ or $\ell$.
In general, this bound is sharp for $\ell \leq n$, e.g.\
when $\sum_{k>n}p_k = (n+1)^{-\alpha}$,
see Corollary~\ref{cor:lower}.

In the null recurrent case $\alpha \leq 1$, we are not aware of any prior results which give a convergence rate.
Our main results in this situation can be summarised as follows.
\begin{enumerate}[label=(\alph*)]
    \item \label{pt:Ornstein} If $\alpha \in (0, 3/2)$, then $\displaystyle \| P^{n+\ell} \delta_0 - P^n \delta_0 \|_{\TV} \lesssim \frac{\ell^{2/3}}{n^{\alpha/3}}$
        for $\ell$ sufficiently small, for example for $\ell \leq n^{3/2 - \alpha} (\log(1 + n))^{-3/2}$.
       See Theorem~\ref{thm:orn}.

     \item \label{pt:mon} If $p_n$ satisfies the \emph{decreasing failure rate} (DFR), i.e.\ the sequence $p_n / (p_n + p_{n+1} + \cdots)$ is non-increasing, then
        $\| P^{n+\ell} \delta_0 - P^n \delta_0 \|_{\TV}$ is given by an explicit expression,
        regardless of the assumption on the tails $\sum_{k > n} p_k$.
        In particular,
        $\| P^{n+1} \delta_0 - P^n \delta_0 \|_{\TV} = 2 \sum_{k > n+1} p_k$.
        See Theorem~\ref{thm:DFR} and Lemma~\ref{lem:HF}.

    \item \label{pt:gen} Assume a corresponding lower bound $\sum_{k>n}p_k \geq C^{-1}n^{-\alpha}$.
    Then, for $\alpha \in (1/2, 1)$,
    \begin{equation*}
    \| P^{n+\ell} \delta_0 - P^n \delta_0 \|_{\TV} \lesssim \frac{\ell^{2 \alpha - 1}}{n^{2 \alpha - 1}}\;.
    \end{equation*}
    Assume furthermore an upper bound on increments $\sum_{k> n}|p_k - p_{k+1}|
        \lesssim n^{-\alpha-1}$.
        Then, for $\alpha \in (0,1)$,
        \begin{equation}
            \label{eq:best1}
            \| P^{n+\ell} \delta_0 - P^n \delta_0 \|_{\TV}
            \lesssim \frac{\ell^\alpha}{n^\alpha}
            \;.
        \end{equation}
        See Theorem~\ref{thm:general}.
\end{enumerate}

\begin{rmk}\label{rem:lower}
If $n\mapsto \sum_{k>n}p_k$ is regularly varying, then~\eqref{eq:best1} is sharp in the sense that
$\| P^{n+\ell} \delta_0 - P^n \delta_0 \|_{\TV}
            \gtrsim \ell^\alpha n^{-\alpha}$,
        see Corollary~\ref{cor:lower}.
Furthermore, in this case, the obvious upper bound
    $\|P^{n + \ell} \delta_0 - P^n \delta_0\|_{\TV}\lesssim 1$ is sharp for $\ell>n$.
\end{rmk}

\begin{rmk}
The bound~\eqref{eq:best} for $\alpha>1$ follows from the case $\ell=1$ and the triangle inequality.
This is \emph{not} the case for the results in~\ref{pt:Ornstein}-\ref{pt:gen}.
\end{rmk}

\begin{rmk}
    As mentioned in Remark~\ref{rem:motivate_model}, bounds on $\| P^{n+\ell} \delta_0 - P^n \delta_0 \|_{\TV}$
    translate directly to memory loss rates for general Harris chains.
    As a simple example, consider a Markov chain $X$ on a measurable space $\Omega$
    with a recurrent point $x\in\Omega$.
    Define $\tau = \inf\{k\geq1\,:\, X_k=x\}$
    and $p_k = \P( \tau = k \mid X_0 = x)$.
    Suppose that
    $\| P^{n+\ell} \delta_0 - P^n \delta_0 \|_{\TV}
    \lesssim
    \ell^\beta n^{-\alpha}$
    for some $\beta\geq\alpha\geq0$ uniformly in $\ell,n\geq 1$,
    where $P$ is the transition operator for the Markov chain~\eqref{eq:RMC} built with these $p_k$.
    For $\mu\neq \nu\in\cP(\Omega)$, define $\|\mu-\nu\|_\beta = \E_{X_0\sim\lambda} \tau^\beta$ where $\lambda=|\mu-\nu|(\Omega)^{-1}|\mu-\nu| \in \cP(\Omega)$. 
    Then $\|Q^n\mu - Q^n\nu\|_{\TV} \lesssim n^{-\alpha}\|\mu-\nu\|_{\beta}$ where $Q$ is the transition operator of $X$.
    See Proposition~\ref{prop:harris_mixing} for details.
\end{rmk}

Our proofs of~\ref{pt:Ornstein}-\ref{pt:gen} are self-contained and independent of each other.
Moreover, the proofs are elementary yet at times delicate, using only
basic facts about regularly varying functions, Hoeffding's inequality, etc.
We make the following additional remarks.
\begin{itemize}
    \item[\ref{pt:Ornstein}] is based on the famous and elegant coupling of Ornstein~\cite{Ornstein_69_RW1,Ornstein_69_RW2},
        where memory loss is drawn from the times when a symmetric random walk on $\Z$ reaches $0$.
        The proof is short, the assumptions are very general, yet the rate is weaker than in~\ref{pt:mon} or~\ref{pt:gen}, and than in~\eqref{eq:best} for $\alpha>1$.
    \item[\ref{pt:mon}] has an even shorter, essentially algebraic proof, although the assumption is very restrictive.
        The decreasing failure rate is used to verify that the renewal sequence $u_n = \P(X_n = 0 \mid X_0 = 0)$
        is monotone, which in turn allows completely bypassing otherwise substantial difficulties.
        Such monotonicity conditions were considered in~\cite{Isaac_1988_rates, Hansen_Frenk_91_monotonicity}.
        We highlight, however, that we give a simple and possibly new proof of the fact that a decreasing failure rate implies
        monotonicity of $u_n$ (see Lemma~\ref{lem:HF}).
    \item[\ref{pt:gen}] is, we believe, the main contribution of this paper.
        The rate~\eqref{eq:best1} is sharp (see Remark~\ref{rem:lower}) while the assumptions
        are still practical to verify.
        Our proof of~\ref{pt:gen} is the longest part of the paper and is based on analysing
        the renewal sequence $u_n$ using Tauberian arguments, inspired by but not relying on~\cite{GL62}.
        Some key differences are that we bound the increments $|u_n-u_{n+\ell}|$ rather than $u_n$ itself,
        and that we require only upper/lower bounds on the tails $\sum_{k>n}p_k$ instead of assuming
        that the tails are regularly varying.
        (Regular variation of $\sum_{k>n}p_k$ is a standard assumption in renewal theory with infinite mean, see e.g.\ \cite{GL62,Teugels_68_renewal,Eri70, Frenk_82_renewal, Doney_97_renewal_infinite,Caravenna_Doney_19,Berger19_Cauchy},
and we are not aware of any results that do not use this assumption.)
        See Remark~\ref{rem:generalisations} for possible generalisations of~\ref{pt:gen}.
\end{itemize}

We conclude the introduction by outlining the motivation for our study.
Mixing rates are a fundamental property of Markov chains,
as evidenced by the body of results in the positive recurrent case,
thus we believe our findings carry intrinsic interest.

Our motivation originates in smooth ergodic theory, where so-called \emph{non-uniformly expanding}
and \emph{non-uniformly hyperbolic} dynamical systems have an intimate relation to
Markov chains similar to~\eqref{eq:RMC}. Various statistical properties have been proved by
representing these systems as \emph{Young towers}, topological Markov chains with
some nice structure~\cite{Young98,Young99}. Further, connections with
renewal-type Markov chains have been established via Lipschitz continuous and measure preserving
semi-conjugacies~\cite{Kor18,CDKM24}.

The landscape is much like in the probabilistic literature. Memory loss is known to
occur in very general conditions~\cite{Lin71}, see~\cite[Theorem~1.3.3]{Aar97}.
In the positive recurrent case, sharp rates of memory loss have been obtained
using coupling~\cite{Young98,Young99,KorLep21} (c.f.~\cite{Pit74,Lind79})
and renewal theory~\cite{Sar02,Gou04}. In the null recurrent case,
renewal sequences have been accurately estimated~\cite{MelTer11,Ter15}.
However, no results currently address rates of memory loss in the null recurrent case.

Extension of our results to ergodic theory will be a subject of a future study.
In particular, the `practical' assumptions in~\ref{pt:gen} have been chosen with such applications in mind.

\subsubsection*{Organization of the paper}

In Section~\ref{sec:model} we state our main results for the Markov chain~\eqref{eq:RMC}.
In Section~\ref{sec:general} we relate memory loss rates of~\eqref{eq:RMC} to those of Markov chains with a recurrent atom
and, more generally, Harris chains.
In Sections~\ref{sec:Ornstein}-\ref{sec:proof_general} we prove the results from Section~\ref{sec:model}, and in Appendix~\ref{sec:Lind}
we justify~\eqref{eq:best}.

\section{Model chain}
\label{sec:model}

In this section we work with Markov chain $X_n$ from~\eqref{eq:RMC} and
initial measures $\delta_0 $ and $P^{\ell} \delta_0$ with some $\ell \geq 1$.
Informally, we are trying to understand how quickly the Markov chain~\eqref{eq:RMC} forgets a time shift.

\subsection{Notation and preliminaries}
\label{sec:not}

Recall that we suppose that $X_n$ is aperiodic and recurrent.
That is, $\sum_{n \geq 1} p_n = 1$ and there exist $c_\# > 0$ and $N_\# \geq 1$ so that
\begin{equation}\label{eq:aperiodic}
    \gcd \{1\leq k \leq N_\# \,:\, p_k \geq c_\# \} = 1
    \;.
\end{equation}
If $\sum_{n \geq 1} n p_n < \infty$, then $X_n$ is positive recurrent, and $X_n$ is null recurrent otherwise.

Let $u_n$, $n \geq 0$, denote the renewal probabilities
\begin{equation*}
    u_n = \P(X_n=0 \mid X_0=0)
    \;.
\end{equation*}
Equivalently, we can define $u_n$ by $u_0 = 1$ and the \emph{renewal equation}
\begin{equation}
    \label{eq:pu}
    u_n = \sum_{k=1}^n p_k u_{n-k}
    \;, \quad
    n \geq 1
    \;.
\end{equation}
For $n \geq 0$ let
\begin{equation*}
    z_n = \sum_{k > n} p_k
    \;.
\end{equation*}
\begin{rmk}
It is a direct verification that $z_n = z_{n-1} (1 - q_{n-1}) = (1-q_0)\cdots (1-q_{n-1})$ for $n\geq 1$.
In particular $z_nq_n = p_{n+1}$ for $n\geq 0$ and $\sum_{n \geq 0} z_n q_n = 1 = z_0$. That is,
$z_n$ are the weights of the unique (up to scaling) stationary measure on $\Omega$.
\end{rmk}

Our notation allows a convenient expression
\begin{equation*}
    P^n \delta_0
    = \sum_{m=0}^{n} u_{n-m} z_m \delta_m
    \;.
\end{equation*}
Subsequently,
\begin{equation}
    \label{eq:mu_diff}
    P^{n+\ell} \delta_0 - P^n \delta_0
    = \sum_{m=n+1}^{n+\ell} u_{n+\ell-m} z_{m} \delta_{m}
    + \sum_{m=0}^{n} (u_{n+\ell-m} - u_{n-m}) z_m \delta_m
    \;.
\end{equation}
The problem of bounding $\|P^{n+\ell} \delta_0 - P^n \delta_0\|_{\TV}$
is therefore closely tied to the asymptotics of the increments of the renewal sequence $u_n$.

We use $X\lesssim Y$ to denote $X\leq C Y$ for a constant $C>0$ that does not depend on $X,Y$.
We write $X\asymp Y$ if $X\lesssim Y$ and $Y\lesssim X$.
We also write $Y\gg X$ whenever $Y\geq C X$ for some $C$ sufficiently large.
For $x\in\R$, we let $\floor x\in\Z$ denote the floor of $x$.

For probability measures $\mu,\nu$ on a general measurable space $(\Omega,\Sigma)$, recall the identity
\begin{equation}\label{eq:TV_identity}
\|\mu-\nu\|_{\TV} = 2\sup_{A\in \Sigma} |\mu(A)-\nu(A)|\;.
\end{equation}

\subsubsection{Regularly varying functions}
\label{subsec:RV}

Recall that a function $\rho \colon [a,\infty)\to(0,\infty)$ is called \emph{regularly varying (RV) with index $\alpha\in\R$} if $\lim_{x\to\infty} \rho(\lambda x)/\rho(x) = \lambda^\alpha$ for all $\lambda>0$.
A RV function with index $0$ is called slowly varying.
For more details on RV functions, we refer to~\cite{bingham1989regular}.
We will use the following basic facts without explicit mention.
Suppose $\rho$ is RV with index $\alpha$.
\begin{itemize}
\item One has $\rho(x) = x^\alpha L(x)$ where $L$ is slowly varying.
    \item For any compact $I\subset(0,\infty)$, the limit $\lim_{x\to\infty} \rho(\lambda x)/\rho(x) = \lambda^\alpha$ is uniform over $\lambda\in I$.
In particular, for $x,y$ sufficiently large, $\rho(x) \asymp \rho(y)$ whenever $x\asymp y$.
    
    \item \textbf{Karamata's theorem}~\cite[Section~1.5.6]{bingham1989regular}. Suppose $\rho$ is locally bounded.
    If $\alpha>-1$, then $\lim_{x\to\infty} x\rho(x)/\int_a^x \rho(y)\,dy = \alpha+1$.
    If $\alpha < -1$, then $\lim_{x\to\infty} x\rho(x)/\int_x^\infty \rho(y)\,dy = -\alpha-1$.
    If $\alpha=-1$, then $\int_a^x \rho(y)\,dy$ is slowly varying.

    We will often use these facts in the following form: if $\alpha>-1$, then $\int_a^x \rho(y)\,dy \asymp x\rho(x)$ and $\int_0^\delta \rho(x^{-1})^{-1} \, dx \asymp \delta\rho(1/\delta)$
    uniformly in $x$ sufficiently large and $\delta$ sufficiently small.
    Likewise, if $\alpha<-1$, then
    $\int_x^\infty \rho(y)\,dy \asymp x\rho(x)$ and $\int_\delta^{1/a} \rho(x^{-1})^{-1} \, dx \asymp \delta\rho(1/\delta)$.
\end{itemize}
If $\rho$ is defined on $\{k,k+1,\ldots\}$, we say that $\rho$ is RV if $x\mapsto \rho(\floor x)$ is RV.

\subsection{Lower bounds}

We begin with an easy \emph{lower bound} on the memory loss for the Markov chain~\eqref{eq:RMC}.

\begin{lemma}
    \label{lem:lower}
    For all $n,\ell\geq0$
    \begin{equation}
        \label{eq:lower}
        \|P^{n + \ell} \delta_0 - P^n \delta_0\|_{\TV}
        \geq
        2 \sum_{m=n+1}^{n+\ell} u_{n+\ell-m} z_{m}
        \;.
    \end{equation}
\end{lemma}

\begin{proof}
    The result follows from~\eqref{eq:mu_diff} and~\eqref{eq:TV_identity} by taking $A=\{n+1,\ldots,n+\ell\}$ in the latter.
\end{proof}

\begin{cor}
    $\|P^{n + 1} \delta_0 - P^n \delta_0\|_{\TV} \geq 2 z_{n+1}$.
\end{cor}

\begin{cor}
    \label{cor:lower}
    Suppose $n\mapsto z_n$ is RV with index $-\alpha$.
    If $\alpha > 1$, then for $n\geq 0$ and $\ell\geq 1$,
    \begin{equation}\label{eq:lower_recur}
        \|P^{n + \ell} \delta_0 - P^n \delta_0\|_{\TV}
        \gtrsim \sum_{m=n+1}^{n+\ell} z_n
        \;.
    \end{equation}
    so in particular, $\|P^{n + \ell} \delta_0 - P^n \delta_0\|_{\TV} \gtrsim \ell z_n$ for $\ell\leq n$.
    If $\alpha \in (0,1)$, then for $n\geq 0$ and $\ell\geq 1$,
    \[
        \|P^{n + \ell} \delta_0 - P^n \delta_0\|_{\TV}
        \gtrsim 1\wedge \frac{z_n}{z_\ell}
        \;.
    \]
    If $\alpha = 1$, then for $\ell \leq n$,
    \[
        \|P^{n + \ell} \delta_0 - P^n \delta_0\|_{\TV}
        \gtrsim
        z_n\sum_{k=1}^{\ell-1} m(k)^{-1}
        \;, \qquad \text{where} \qquad
        m(n) = \sum_{k=0}^{n-1} z_k
        \;.
    \]
\end{cor}

\begin{proof}
    If $\alpha > 1$ then $u_n \gtrsim 1$ for sufficiently large $n$.
    Hence~\eqref{eq:lower_recur} follows from~\eqref{eq:lower} and $\|P^{n + \ell} \delta_0 - P^n \delta_0\|_{\TV} \gtrsim \ell z_n$ for $\ell\leq n$ follows since $z$ is RV.

    Suppose $\alpha\in(0,1)$.
    Then by~\cite[Theorem~1.1]{GL62}, $\liminf_{n\to\infty} n z_n u_n > 0$.
    Let $r>0$ such that $u_n \gtrsim n^{-1}z_n^{-1}$ for all $n\geq r$.
    It suffices to consider $r< n$.
    In this case, for $\ell\leq n$, by~\eqref{eq:lower} and Karamata's theorem,
    \begin{align*}
        \|P^{n + \ell} \delta_0 - P^n \delta_0\|_{\TV}
        & \geq
        \sum_{m=n+1}^{n+\ell} u_{n+\ell-m} z_m
        \asymp
        z_n \sum_{k=0}^{\ell - 1} u_k
         \gtrsim
        \sum_{k=r}^{\ell - 1} \frac{z_n}{k z_k}
        \asymp \frac{z_n}{z_\ell}
        \;,
    \end{align*}
    while for $\ell>n$,
    \begin{align*}
        \|P^{n + \ell} \delta_0 - P^n \delta_0\|_{\TV}
        & \geq
        \sum_{m=n+1}^{n+\ell-r} u_{n+\ell-m} z_m
        \gtrsim 1
        \;,
    \end{align*}
    where we used $u_{n+\ell-m} \gtrsim \ell^{-1} z_\ell^{-1}$
    and $z_m \gtrsim z_\ell$.

    Suppose $\alpha = 1$. By~\cite[Theorem~1]{Eri70},
    $\lim_{n \to \infty} m(n) u_n = 1$.
    Then there is $r>0$ so that $u_n \geq \frac{1}{2} m(n)^{-1}$ for all $n \geq r$
    and, repeating the calculation for $\alpha \in (0,1)$, for $r < \ell \leq n$,
    \begin{align*}
        \|P^{n + \ell} \delta_0 - P^n \delta_0\|_{\TV}
        & \gtrsim
        z_n \sum_{k=0}^{\ell-1} u_k
        \gtrsim
        z_n \sum_{k=r}^{\ell-1} m(k)^{-1}
        \;.
        \qedhere
    \end{align*}
\end{proof}

\begin{rmk}
    Suppose that $\sum_{k > n} p_k \lesssim n^{-\alpha}$ with $\alpha > 0$.
    For $\alpha > 1$ (when $X_n$ is positive recurrent),
    $\|P^{n + 1} \delta_0 - P^n \delta_0\|_{\TV} \lesssim n^{-\alpha}$ by~\eqref{eq:best}.
    By Corollary~\ref{cor:lower}, this is sharp.
    A question of special interest and the initial motivation for this work was to investigate
    whether the bound $\|P^{n + 1} \delta_0 - P^n \delta_0\|_{\TV} \lesssim n^{-\alpha}$
    can be extended to $\alpha \leq 1$.
\end{rmk}

\subsection{Upper bounds}

Under very relaxed assumptions, using only an upper bound on the tails $\sum_{k > n} p_k$, we obtain:
\begin{thm}
    \label{thm:orn}
    Suppose that $\sum_{k > n} p_k \leq \rho(n)$, where $\rho$ is RV with index $-\alpha \in (-3/2, 0)$.
    Then
    \begin{equation}
        \label{eq:orn}
        \| P^{n+\ell}\delta_0 - P^n\delta_0 \|_{\TV}
        \lesssim \ell^{2/3} \rho(n)^{1/3}
    \end{equation}
    uniformly in $n \geq 1$ and $\ell \geq 1$ restricted by:
    \begin{itemize}
        \item $\ell \leq \rho(n)n^{3/2} M(n)^{-3/2}$ with $M(n) = 1+\int_1^n \rho(t)\,dt$ if $\alpha = 1$,
        \item $\ell \leq \rho(n)n^{3/2}$ if $\alpha \in (1, \frac32)$.
    \end{itemize}
    In all cases, the proportionality constant depends only on the constants in~\eqref{eq:aperiodic} and on $\rho$.
\end{thm}

We prove Theorem~\ref{thm:orn} in Section~\ref{sec:Ornstein}.
The proof is based on quantifying an elegant coupling of Ornstein~\cite{Ornstein_69_RW1,Ornstein_69_RW2}.

\begin{rmk}
    For $\rho(n) = C n^{-\alpha}$ with $\alpha > 1$, the upper bound~\eqref{eq:orn}
    is significantly weaker than~\eqref{eq:best} that works under the same assumptions,
    and this is an inherent limitation of the proof.
    Moreover, for $\alpha \leq 1$, the upper bound~\eqref{eq:orn} is far from the lower bound~\eqref{eq:lower}.
    However, this is the only rate we obtain that works under the minimal assumption $\sum_{k > n} p_k \leq \rho(n)$.
    It remains an open question if~\eqref{eq:orn} can be improved without making extra assumptions.
\end{rmk}

Under a very restrictive assumption we show that the lower the bound~\eqref{eq:lower} is precise:
\begin{thm}
    \label{thm:DFR}
    Assume that $u_n$ is non-increasing with $n$.
    Then
    \begin{equation*}
        \|P^{n + \ell} \delta_0 - P^n \delta_0\|_{\TV}
        =
        2 \sum_{m=n+1}^{n+\ell} u_{n+\ell-m} z_{m}
        \;.
    \end{equation*}
\end{thm}

We prove Theorem~\ref{thm:DFR} in Section~\ref{sec:monotone}.
The following lemma provides a simple way to verify the monotonicity condition of Theorem~\ref{thm:DFR}.

\begin{lemma}
    \label{lem:HF}
    Assume \emph{decreasing failure rate} (DFR), i.e.\ that the sequence
    $p_n / (p_n + p_{n+1} + \cdots)$ is non-increasing with $n$.
    Then $u_n$ is non-increasing in $n$.
\end{lemma}

\begin{proof}
The lemma appears in~\cite{Hansen_Frenk_91_monotonicity} but we provide a different, simple proof in Section~\ref{sec:HF}.
\end{proof}

Finally, we show that the optimal rates are achievable for null recurrent chains under \emph{practical} assumptions:

\begin{thm}\label{thm:general}
    Suppose that
    \begin{equation}\label{eq:p_bound_imp}
        c \rho (n) \leq z_n \leq \rho(n)
    \end{equation}
    for all $n \geq 1$, where $c \in (0,1]$ and $\rho$ is regularly varying with index $-\alpha \in (-1,0)$.

    If $\alpha\in (\frac12,1)$, then uniformly in $n\geq \ell \geq 1$,
    \begin{equation}\label{eq:some_rate}
        \|P^{n + \ell} \delta_0 - P^n \delta_0\|_{\TV}
        \lesssim
        \frac{n \rho(n)^2}{\ell \rho(\ell)^{2} }\;.
    \end{equation}
    If $\alpha \in (0,1)$ and, in addition to~\eqref{eq:p_bound_imp}, we assume
    \begin{equation}\label{eq:p_reg}
        \sum_{k> n}|p_k - p_{k+1}|
        \leq c^{-1} n^{-1}\rho(n)
        \;,
    \end{equation}
    then, uniformly in $n\geq\ell\geq 1$,
    \begin{equation}\label{eq:mu_reg}
        \|P^{n + \ell} \delta_0 - P^n \delta_0\|_{\TV}
        \lesssim \frac{\rho(n)}{\rho(\ell)}
        \;.
    \end{equation}
    All the proportionality constants depend only on the constants in~\eqref{eq:aperiodic} and on $c,\rho$.
\end{thm}

We prove Theorem~\ref{thm:general} in Section~\ref{sec:proof_general}.
Remark that~\eqref{eq:mu_reg} is generally sharp by Corollary~\ref{cor:lower}.

\begin{rmk}\label{rem:generalisations}
    The only place where we use the lower bound on $z_n$ is Lemma~\ref{lem:1-P}.
    It is possible to assume a bound of the form
    $C_1 n^{-\beta} \leq z_n \leq C_2 n^{-\alpha}$ with possibly different exponents $\beta \geq \alpha$,
    but then the bound on $\|P^{n+\ell} \delta_0 - P^n \delta_0\|_{\TV}$ becomes more complicated to state.

    In the same spirit, one can generalise the condition~\eqref{eq:p_reg}
    to $\sum_{k \geq n} |p_{k}+p_{k+1}|\lesssim \hat\rho(n)$ for another RV function $\hat \rho$,
    but again with a more complicated final result.
\end{rmk}

\begin{rmk}
Condition~\eqref{eq:p_reg} implies $p_n = \sum_{m\geq n} ( p_m - p_{m+1} ) \lesssim n^{-1}\rho(n)$, which,
by Karamata's theorem, implies $\sum_{m> n} p_m \lesssim \rho(n)$.
Hence~\eqref{eq:p_reg} implies the upper bound in~\eqref{eq:p_bound_imp} (up to a multiplicative constant),
and so~\eqref{eq:p_reg} is the most `optimistic' bound that is consistent with the upper bound in~\eqref{eq:p_bound_imp},
i.e.\ if $n^{-1}\rho(n)$ is replaced by $\bar\rho(n)$ for RV $\bar\rho$ with $\bar\rho(n) \lesssim n^{-1}\rho(n)$ then
we would have $\sum_{m\geq n} p_m \lesssim n\bar\rho(n)$,
which is stronger than the upper bound from~\eqref{eq:p_bound_imp}.
\end{rmk}

\section{General Markov chains}
\label{sec:general}

In this section we estimate memory loss on Harris chains,
using the bounds of $\|P^n(\delta_0 - P^\ell \delta_0)\|_{\TV}$ for the model chain~\eqref{eq:RMC}.
This is our principle motivation for studying the chain~\eqref{eq:RMC} and for the results of Section~\ref{sec:model}.
First we treat Markov chains with a recurrent atom in Section~\ref{sec:atom}
and then general Harris chains in Section~\ref{sec:Harris}.
We discuss examples in Section~\ref{sec:har-examples}.

Given $p_n \geq 0$ with $\sum_{n \geq 1} p_n = 1$ we denote
\begin{equation}
    \label{eq:bn}
    b_\ell(n)
    = \| P^n (\delta_0 - P^\ell \delta_0) \|_{\TV}
    \;,
\end{equation}
where $P$ is the transition operator on the Markov chain~\eqref{eq:RMC} defined by $p_n$.

Throughout this section, let $(\Omega,\Sigma)$ be a measurable space
and $X_n$, $n\geq 0$, a Markov chain with state space $\Omega$
and transition operator $Q\colon \cP(\Omega)\to\cP(\Omega)$.

\subsection{Markov chains with a recurrent atom}
\label{sec:atom}

Suppose that $X_n$ has a recurrent atom $S \in \Sigma$.
That is,
there exists a probability measure $\beta$ on $\Omega$ such that
\begin{enumerate}[label=(\alph*)]
    \item $\P(X_n \in S \text{ infinitely often} \mid X_0 = x) = 1$ for each $x \in \Omega$, and
    \item $\P(X_{n+1} \in A \mid X_n = x) = \beta(A)$ for all $A \in \Sigma$ and $x \in S$.
\end{enumerate}
Define
\begin{equation}
    \label{eq:gen-pn}
    \tau = \inf \{k \geq 1 \,:\, X_{k-1} \in S \}
    \qquad \text{and} \qquad
    p_n = \P(\tau = n \mid X_0 \sim \beta)
    \;.
\end{equation}
Since $X_n$ is recurrent, $\sum_{n\geq 1} p_n = 1$.

Let $\beta_m$, $m \geq 0$, be probability measures on $\Omega$ given by
\begin{align*}
    \beta_m (A)
    & = \P(X_{m+1} \in A \mid X_0 \in S \text{ while } X_1, \ldots, X_m \notin S)
    \\
    & = \P(X_m \in A \mid X_0 \sim \beta \text{ and } \tau \geq m+1)
    \;.
\end{align*}
For a measure $\zeta$ on $\{0,1, \ldots\}$, let $\pi \zeta$ be its projection on $\Omega$ obtained by
\[
    \pi \zeta
    = \sum_{m \geq 0} \zeta(m) \beta_m
    \;.
\]

The following lemma links $Q^n \mu$, for $\mu \in \cP(\Omega)$, with
the evolution of $P^n \delta_0$ on~\eqref{eq:RMC} for which
the $p_n$ are given by~\eqref{eq:gen-pn}, up to a small remainder.

\begin{lemma}
    \label{lem:gen1}
    Let $X_0 \sim \mu \in \cP(\Omega)$ and $A \in \Sigma$. Then
    \[
        \P(X_n \in A)
        = Q^n \mu (A)
        = \pi \Bigl[ \sum_{t=1}^n \P(\tau = t) P^{n-t} \delta_0 \Bigr] (A)
        +
        \P( X_n \in A \,, \tau > n)
        \;.
    \]
\end{lemma}

\begin{proof}
    Let
    \[
        Z_n = \begin{cases}
            -1 & \text{if } n = 0 \text{ or } X_0, \ldots X_{n-1} \notin S \;, \\
            \min \{ k \geq 0 \,:\, X_{n-1-k} \in S \} & \text{otherwise} \;.
        \end{cases}
    \]
    Then $(X_n, Z_n)$ is a Markov chain on $\Omega \times \{-1,0,\ldots\}$, and, for $m\geq 0$,
    \[
        \beta_m (A) = \P(X_n \in A \mid Z_n = m)
        \;, \qquad
        \tau = \inf \{k \geq 1 \,:\, Z_k = 0\}
        \;.
    \]
    Observe that $Z_\tau,Z_{\tau+1},\ldots$ behaves as the Markov chain~\eqref{eq:RMC} defined by $p_n$ in~\eqref{eq:gen-pn},
    with the initial condition $Z_\tau = 0$. Consequently,
    \[
        \sum_{m=0}^n \P(Z_{t + n} = m \mid Z_t = 0) \beta_m
        = \pi P^n \delta_0
        \;.
    \]
    Suppose that $X_0 \sim \mu$.
    For $A \in \Sigma$, $n \geq t\geq 1$, $0 \leq m < n$, we have
    \begin{align*}
        \P( X_n \in A ,\, Z_n = m ,\,\tau = t)
        & = \P(X_n \in A \mid Z_n = m ,\, \tau = t) \P(Z_n = m \mid \tau = t) \P(\tau = t)
        \\
        & = \beta_m(A) \P(Z_n = m \mid Z_t = 0) \P(\tau = t)
        \;.
    \end{align*}
    Therefore,
    \begin{align*}
        \P(X_n \in A)
        & = \sum_{t=1}^n \sum_{m=0}^{n-t} \P( X_n \in A ,\, Z_n = m ,\,\tau = t)
        + \P(X_n \in A ,\, Z_n = -1)
        \\
        & = \sum_{t=1}^n \sum_{m=0}^{n-t} \beta_m(A) \P(Z_n = m \mid Z_t = 0) \P(\tau = t)
        + \P(\tau > n, X_n\in A)
        \\
        & = \sum_{t=1}^n \P(\tau = t) \pi [P^{n-t} \delta_0] (A)
        + \P(\tau > n, X_n\in A)
        \;.
        \qedhere
    \end{align*}
\end{proof}

Lemma~\ref{lem:gen1} can be used to estimate memory loss, for example:

\begin{cor}
    \label{cor:gen}
    Let $\mu, \mu' \in \cP(\Omega)$. Then
    \[
        \| Q^n (\mu - \mu') \|_{\TV}
        \leq \sum_{t=1}^n \bigl| \mu(\tau = t) - \mu'(\tau = t) \bigr| \, b_{t-1} (n-t)
        + 2 \mu(\tau > n) + 2 \mu'(\tau > n)
        \;,
    \]
    where $b_\ell(n)$ are as in~\eqref{eq:bn}
    and $\mu(\tau = k)$ is a shorthand for $\P(\tau = k \mid X_0 \sim \mu)$.
\end{cor}

\begin{proof}
    Let $X_0 \in \mu$ and let $X'_n$ be a copy of $X_n$ with $X'_0 \sim \mu'$.
    Let $\tau'$ denote the copy of $\tau$.
    By Lemma~\ref{lem:gen1},
    \begin{equation*}
    \begin{split}
        Q^n \mu(A) - Q^n \mu'(A)
        & = \pi \Bigl[ \sum_{t=1}^n \bigl( \P(\tau = t) - \P(\tau' = t) \bigr) P^{n-t} \delta_0 \Bigr] (A)
        \\
        & \qquad + \P( X_n \in A \,, \tau > n) - \P( X'_n \in A \,, \tau' > n)
        \;.
    \end{split}
    \end{equation*}
    Furthermore $\sum_{t=1}^\infty \bigl(\P(\tau = t) - \P(\tau' = t)\bigr) P^{n-1} \delta_0 = 0$ and thus
    \begin{equation*}
        \begin{aligned}
            \sum_{t=1}^n \bigl( \P(\tau = t) - \P(\tau' = t) \bigr) P^{n-t} \delta_0
            & = \sum_{t=1}^n \bigl( \P(\tau = t) - \P(\tau' = t) \bigr) (P^{n-t} \delta_0 - P^{n-1} \delta_0)
            \\
            & \quad
            - (\P(\tau > n) - \P(\tau'>n)) P^{n-1} \delta_0
            \;.
        \end{aligned}
    \end{equation*}
    The result follows readily.
\end{proof}

Corollary~\ref{cor:gen} can be strengthened when $\mu$ and $\mu'$ overlap, i.e.\ when $|\mu - \mu'|(\Omega) < 2$.
For $\mu, \mu' \in \cP(\Omega)$ let
\begin{equation}
    \label{eq:mumu}
    |\mu - \mu'|(\tau = n)
    = |\mu - \mu'|(\Omega) \, \P(\tau = n \mid X_0 \sim \lambda)
    \;,
\end{equation}
where $\lambda$ is the normalized to probability difference $|\mu - \mu'|$, namely
$\lambda = \frac{|\mu - \mu'|}{|\mu - \mu'|(\Omega)}$, if $\mu \neq \mu'$, and an arbitrary measure
otherwise.

\begin{cor}
    \label{cor:genn}
    Let $\mu, \mu' \in \cP(\Omega)$. Then
    \[
        \| Q^n (\mu - \mu') \|_{\TV}
        \leq \sum_{t=1}^n |(\mu - \mu') (\tau = t) | \, b_{t-1} (n-t)
        + 2 |\mu - \mu'|(\tau > n)
        \;,
    \]
    where $b_\ell(n)$ are as in~\eqref{eq:bn}
    and $|\mu - \mu'|(\tau = k)$ is as in~\eqref{eq:mumu}.
\end{cor}

\begin{proof}
    Write $\mu - \mu' = \frac12|\mu - \mu'|(\Omega) \, (\kappa - \kappa')$, where $\kappa$ and $\kappa'$ are mutually
    singular probability measures.
    Then $\lambda = \frac{1}{2}|\kappa-\kappa'| = \frac{1}{2}(\kappa+\kappa')$ and $\P(\tau=n\mid X_0\sim\lambda)= \frac{1}{2}(\kappa(\tau=n)+\kappa'(\tau=n))$.
    The result follows from Corollary~\ref{cor:gen} applied to $\kappa, \kappa'$.
\end{proof}

\subsection{Harris chains}
\label{sec:Harris}

Suppose now that $X_n$ is a \emph{Harris chain}~\cite{Lindvall_02_coupling, Durrett_19}, i.e.\
there exists $S \in \Sigma$, $\eps > 0$ and a probability measure $\beta$ on $\Omega$ such that
\begin{enumerate}[label=(\alph*)]
    \item $\P(X_n \in S \text{ infinitely often} \mid X_0 = x) = 1$ for each $x \in \Omega$, and
    \item $\P(X_{n+1} \in A \mid X_n = x) \geq \eps \beta(A)$ for all $A \in \Sigma$ and $x \in S$.
\end{enumerate}

\begin{rmk}
    If $\eps = 1$, then $S$ is an atom and we can work with $X_n$ as in Section~\ref{sec:atom}.
    For $\eps < 1$ the situation is not too different: using Nummelin's splitting technique~\cite{Nummelin_78_splitting}
    we can extend $X_n$ to a Markov chain with an atom.
\end{rmk}

\begin{rmk}
    There are different definitions of Harris (recurrent) chains in the literature, see e.g.~\cite{Nummelin_84_MC,Meyn_Tweedie_09_MC,Benaim_Hurth_2022_MC_Metric}.
    We work with a definition that is simple for our purposes.
    See e.g.~\cite{Nummelin_Tuominen_82_geom_erg}
    that relates our definition to others.
\end{rmk}

Let $\xi_n$, $n \geq 0$, be a sequence of independent (also from the Markov chain $X_n$) Bernoulli random variables
with $\P(\xi_n = 1) = \eps$ and $\P(\xi_n=0)=1-\eps$.
Denote
\begin{equation}\label{eq:tau_def_xi}
    \tau = \inf \{k \geq 1 \,:\, X_{k-1} \in S \text{ and } \xi_{k-1} = 1 \}
    \qquad \text{and} \qquad
    p_n = \P( \tau = n \mid X_0 \sim \beta)
    \;.
\end{equation}
This way, $p_n$ and $\tau$ are geometrically compounded versions of~\eqref{eq:gen-pn}.
Again, $\sum_{n\geq 1} p_n = 1$.

Let $\beta_m$, $m \geq 0$, be probability measures on $\Omega$ given by
\[
    \beta_m (A)
    = \P(X_m \in A \mid X_0 \sim \beta \text{ and } \tau \geq m+1)
    \;.
\]
For a measure $\zeta$ on $\{0,1, \ldots\}$, let $\pi \zeta$ be the measure on $\Omega$ obtained by
\[
    \pi \zeta
    = \sum_{m \geq 0} \zeta(m) \beta_m
    \;.
\]
The exact copy of Lemma~\ref{lem:gen1} holds:

\begin{lemma}
    \label{lem:har}
    Let $X_0 \sim \mu \in \cP(\Omega)$ and $A \in \Sigma$. Then
    \[
        \P(X_n \in A)
        = Q^n \mu (A)
        = \pi \Bigl[ \sum_{t=1}^n \P(\tau = t) P^{n-t} \delta_0 \Bigr] (A)
        +
        \P( X_n \in A ,\, \tau > n)
        \;.
    \]
\end{lemma}

\begin{proof}
    The strategy, following~\cite{Nummelin_78_splitting}, is to extend $X_n$ to a Markov chain with an atom and then apply Lemma~\ref{lem:gen1}.
    To this end, let $H(x, A) = \P(X_{n+1} \in A \mid X_n = x)$ denote the kernel of $X_n$.
    Define $H_0(x,A)$ by
    \[
        H(x, A) = \begin{cases}
            \eps \beta(A) + (1 - \eps) H_0(x, A) & \text{if } x \in S \; , \\
            H_0 (x, A) & \text{otherwise} \;,
        \end{cases}
    \]
    where, if $\eps = 1$, we set $H_0(x,\cdot)$ arbitrarily for $x\in S$.

    Let $\hX_n$ be the Markov chain on $\hOmega = \Omega \times \{0,1\}$ with kernel
    \begin{equation*}
        \hH ((x,0), (A, t))
        = \begin{cases}
            (1 - \eps) H_0(x, A \cap S) + H_0(x, A \setminus S) & \text{if } t = 0 \;, \\
            \eps H_0(x, A \cap S) & \text{if } t = 1 \;,
        \end{cases}
    \end{equation*}
    \begin{equation}
    \hH ((x,1), (A, t))
        =
        \hbeta(A,t)
        := \begin{cases}
            (1-\eps) \beta (A \cap S) + \beta(A \setminus S) & \text{if } t = 0 \;, \\
            \eps \beta(A \cap S) & \text{if } t = 1 \;.
        \end{cases}
        \label{eq:hbeta_def}
    \end{equation}
    Recall that $X_0 \sim \mu$ and let $ \hX_0 \sim \hmu$ where
    \[
        \hmu(A \times \{0\}) = \mu(A \setminus S) + (1 - \eps) \mu(A \cap S)
        \quad \text{and} \quad
        \hmu(A \times \{1\}) = \eps \mu(A \cap S)
        \;.
    \]
    Observe that $\P(\hX_n = (x, 1) \mid \hX_n \in \{x\} \times \{0,1\}) = \eps \bone_{x \in S}$.
    It is then a direct verification that
    \[
        \P \bigl( \hX_{n+1} \in A \times \{0,1\} \mid \hX_n \in \{x\} \times \{0,1\} \bigr)
        = \P(X_{n+1} \in A \mid X_n = x)
        \;.
    \]
    That is, the first component marginal process of $\hX_n$ has the same law as $X_n$.
    Moreover $\hS = S \times \{1\}$ is an atom for $\hX_n$, which thus falls within the framework of Section~\ref{sec:atom}
    with $\htau = \inf \{k \geq 1 \,:\, \hX_{k-1} \in \hS \}$, $\hbeta\in\cP(\hOmega)$ as in~\eqref{eq:hbeta_def},
    and the corresponding projection $\hpi$.
    Moreover, $\P(\tau = n \mid X_0 \sim \beta) = \P(\htau = n \mid \hX_0 \sim \hbeta)$,
    so the distributions $p_n$ for $X_n$ and $\hX_n$ coincide.
    Thus, by Lemma~\ref{lem:gen1}, writing $\hA = A \times \{0,1\}$,
    \[
        \P(\hX_n \in \hA)
        = \hpi \Bigl[ \sum_{t=1}^n \P(\htau = t) P^{n-t} \delta_0 \Bigr] (\hA)
        +
        \P( \hX_n \in \hA ,\, \htau > n)
        \;.
    \]
    The hats can be removed, yielding the desired result.
\end{proof}

With the same proof as Corollary~\ref{cor:genn}, we obtain the following.

\begin{cor}
    \label{cor:har}
    Let $\mu, \mu' \in \cP(\Omega)$. Then
    \[
        \| Q^n (\mu - \mu') \|_{\TV}
        \leq \sum_{t=1}^n |(\mu - \mu') (\tau = t)| \, b_{t-1} (n-t)
        + 2 |\mu - \mu'|(\tau > n)
        \;,
    \]
    where $b_\ell(n)$ are as in~\eqref{eq:bn}
    and $|\mu - \mu'|(\tau = k)$ is as in~\eqref{eq:mumu}.
\end{cor}

\subsection{An example}
\label{sec:har-examples}

Suppose we are in the setting of Section~\ref{sec:Harris}. In particular, $X$ is a Harris chain.
Recall~\eqref{eq:mumu} where $\tau$ is from~\eqref{eq:tau_def_xi}. Corollary~\ref{cor:har} provides a convenient way to estimate memory loss:

\begin{prop}\label{prop:harris_mixing}
Suppose that
$
b_\ell(n) = \|P^n\delta_0-P^{n+\ell} \delta_{0}\|_{\TV}
\lesssim
\rho(n)g(\ell)
$
uniformly in $n\geq\ell\geq 1$,
where $\rho,g$ are RV with $g$ increasing.
    Denote
    \begin{equation}\label{eq:theta_norm}
        \|\mu-\mu'\|_{g}
        = \sum_{\ell \geq 0} g(\ell) |\mu-\mu'|(\tau =\ell)
        = |\mu - \mu'|(\Omega) \, \E_{X_0\sim \lambda} g(\tau)\;,
    \end{equation}
    where $\lambda = \frac{|\mu - \mu'|}{|\mu - \mu'|(\Omega)}$.
    Then
    \begin{equation}\label{eq:rates_Harris}
        \|Q^n\mu - Q^n\mu'\|_{\TV} \lesssim \{\rho(n)+1/g(n)\}\|\mu-\mu'\|_{g}\;.
    \end{equation}
\end{prop}

\begin{proof}
    Denote $\mu_\ell = \mu(\tau=\ell), \mu'_\ell = \mu'(\tau=\ell)$.
    Remark that
    \begin{equation*}
        | \mu_\ell-\mu'_\ell |
        = \Bigl| \int_{\Omega} \P(\tau=\ell\mid X_0=x) (\mu-\mu')(dx) \Bigr|
        \leq |\mu-\mu'|(\tau = \ell)\;.
    \end{equation*}
    By Corollary~\ref{cor:har},
    \begin{equation*}
        \| Q^n (\mu - \mu') \|_{\TV}
        \leq \sum_{\ell=1}^n \bigl| \mu_\ell - \mu'_\ell \bigr| \, b_{\ell-1} (n-\ell)
        + 2 |\mu-\mu'|(\tau > n)
        \;.
    \end{equation*}
    Then, using $b_\ell(n) \lesssim 1\wedge g(\ell) \rho(n)$,
    \begin{align*}
        \sum_{\ell=1}^{n} |\mu_{\ell}-\mu'_\ell| b_{\ell-1}(n-\ell)
        &\lesssim
        \sum_{n/2<\ell\leq n} |\mu_{\ell}-\mu'_\ell|
        +
        \sum_{1\leq \ell \leq n/2} |\mu_{\ell}-\mu'_\ell| g(\ell) \rho(n-\ell)
        \\
        &\lesssim (g(n)^{-1} + \rho(n))\|\mu-\mu'\|_g\;,
    \end{align*}
    where we used $\sum_{n/2<\ell \leq n} |\mu_{\ell}-\mu'_\ell| \lesssim g(n)^{-1}\|\mu-\mu'\|_g$.
    Finally, since $g$ is increasing, we clearly have $|\mu-\mu'|(\tau>n) \leq g(n)^{-1}\|\mu-\mu'\|_g$.
\end{proof}

\begin{rmk}
    Proposition~\ref{prop:harris_mixing} can be seen as a version of classical results for positive recurrent
    Markov chains, like~\cite[Theorem~4.2,~p27]{Lindvall_02_coupling}.
    However, Proposition~\ref{prop:harris_mixing} does not yield sharp rates in some simple situations.
    For example, consider the Markov chain~\eqref{eq:RMC} with $S = \{\circled{0}\}$ and $p_k = C k^{-\alpha-1}$, $\alpha \in (0,1)$,
    so that
    $\sum_{k > n} p_k = \rho(n)$ with $\rho(n) = C'n^{-\alpha}$.
    Then by Theorem~\ref{thm:general}, $b_\ell(n) \lesssim \rho(n) g(\ell)$ with $g(\ell) = \ell^\alpha$,
    and this is sharp by Corollary~\ref{cor:lower}.
    Let $\mu = \delta_0$ and $\mu' = \delta_1$. Then, attempting to apply
    Proposition~\ref{prop:harris_mixing}, we estimate
    $\P(\tau = n \mid X_0 \sim \lambda) \asymp p_n = C n^{-\alpha - 1}$, and hence
    $\E_{X_0\sim \lambda} g(\tau) \asymp \sum_{n\geq 1} n^{-1}$, which is infinite, and thus so is the right-hand side of~\eqref{eq:rates_Harris}.
    This is not surprising because~\eqref{eq:theta_norm} is a strong moment while we otherwise work with weak moments.

    To apply Proposition~\ref{prop:harris_mixing} to this example, we could instead use the bound $b_\ell(n) \lesssim \ell^\beta n^{-\beta}$ for $\beta<\alpha$. Then
    $\E_{X_0\sim \lambda} \tau^\beta \asymp \sum_{n\geq 1} n^{\beta-\alpha-1}<\infty$,
    which yields $\|Q^n\mu - Q^n\mu'\|_{\TV} \lesssim n^{-\beta}$
    but which is not sharp.

    It is possible to use Lemmas~\ref{lem:gen1} or~\ref{lem:har} directly to recover the sharp rate,
    or, of course,
    \[
        \|P^n(\delta_0 - \delta_1)\|_{\TV}
        = (1 - p_1)^{-1} \|P^n(\delta_0 - P \delta_0)\|_{\TV}
        = (1 - p_1)^{-1} b_1(n)
        \lesssim n^{-\alpha}
        \;.
    \]
\end{rmk}

\section{Ornstein's coupling}
\label{sec:Ornstein}

In this section we prove Theorem~\ref{thm:orn}, adapting Ornstein's proof~\cite{Ornstein_69_RW1,Ornstein_69_RW2}
of Orey's theorem. The proof is based on reduction to a symmetric random walk on integers and coupling
when the random walk reaches $0$.

\subsection{Meeting times and coupling inequality}

Consider probability distributions $\mu,\nu$ on a state space $\Omega$ and a transition operator $P\colon \cP(\Omega)\to \cP(\Omega)$.
We recall how to bound $\|P^n \mu - P^n\nu\|_{\TV}$ by coupling.
Consider $\Omega$-valued Markov chains $X,Y$, defined on the same probability space, with initial distributions $X_0\sim \mu,Y_0\sim\nu$.
Suppose that exists a random time $T>0$ such that $X_n = Y_n$ for all $n\geq T$.
Then
\[
\P(X_n \neq Y_n) \leq \P( T > n)
\]
which in particular implies that, for any measurable $A\subset \Omega$,
\begin{align*}
    \P (X_n \in A)
    & = \P(X_n\in A, \, T\leq n) + \P(X_n\in A, \, T>n)
    \leq \P(Y_n\in A, \, T\leq n) + \P(T>n)
    \\
    & \leq \P(Y_n\in A) + \P(T>n)\;.
\end{align*}
By symmetry, $\P (Y_n \in A) \leq
\P(X_n\in A) + \P(T>n)$.
We thus obtain the \emph{coupling inequality,} a simple yet very useful tool
for estimating total variation (see e.g.~\cite[Chapter~I]{Lindvall_02_coupling}):
\begin{equation}\label{eq:TV_coupling}
\|P^n\mu - P^n\nu\|_{\TV}
=
2\sup_A |\P(X_n \in A)-\P(Y_n \in A)| \leq
2\P(T>n)\;.
\end{equation}

\subsection{Proof of Theorem~\ref{thm:orn}}

    By the aperiodicity condition~\eqref{eq:aperiodic}, there exists $A\geq1$ such that $p_A,p_{A+1}>0$.
    Let $X$ be the Markov chain~\eqref{eq:RMC} with $X_0=0$.
    Let $a_0 = \inf\{j\geq \ell \,:\, X_j = 0\}$ and $a_{k+1} = \inf \{j > a_k : X_j = 0 \}$ index visits of $X$ to $0$ after time $\ell$.
    Let $\Delta$ denote backward increments, so $\Delta a_k = a_k - a_{k-1}$
    and $\P(\Delta a_k = n) = p_n$ for all $k\geq 1$.

On the same (possibly enriched) probability space as $X$,
we now construct another Markov chain $Y$ with the same transition probabilities but with initial state $Y_0 = X_\ell \sim P^\ell \delta_0$.
To define $Y$, it suffices to specify the times $b_0 = \inf \{j\geq 0 \,:\,Y_j=0\}$ and $b_{k+1} = \inf \{j>b_{k}\,:\, Y_j=0\}$
that index the visits of $Y$ to $0$.
First, let $b_0 = a_0-\ell$ so that $Y_0 = X_\ell,Y_{1} = X_{\ell+1},\ldots,Y_{b_0} = X_{a_0}$.
Then we define $b_k$ for $k\geq 1$ by
    \[
        \Delta b_k =
        \begin{cases}
            \Delta a_k  & \text{if } a_{k-1} = b_{k-1} \text{ or } \Delta a_k \notin \{p_A,p_{A+1}\} \;, \\
            s_k        & \text{otherwise,}
        \end{cases}
    \]
    where $s_k$ is a sequence of independent (also from $X$) and identically distributed
    random variables with $\P(s_k = A) = p_A / (p_A + p_{A+1})$ and $\P(s_k = A+1) = p_{A+1} / (p_A + p_{A+1})$.
    This completes the construction of $Y$.
    Remark that, for $k\geq 1$, $\Delta b_k$ has the same distribution as $\Delta a_k$, thus $Y$ indeed has the same
    transition operator as $X$.

    Denote $d_k = a_k - b_k$. Let $k_0 = 0$ and let $k_i$, $i \geq 1$, index when $d_k$ can change, namely
    $k_{i+1} = \inf \{j > k_i \,:\, \Delta a_j \in \{p_A,p_{A+1}\}\}$.
    Observe that $\Delta a_{k_i} = a_{k_i} - a_{k_{i-1}}$ are i.i.d.\ with $\P(\Delta a_{k_i} \geq j) \leq \rho(j)$.
    Furthermore, $d_{k_0},d_{k_1},\ldots$ is a symmetric random walk on integers starting at $d_{k_0} = d_0 = \ell \geq 1$
    and stopping when it reaches $0$, with $\Delta d_{k_i} = d_{k_i} - d_{k_{i-1}} \in\{-1,0,1\}$ and
    \[
        \P(\Delta d_{k_i} = \pm 1 \mid d_{k_{i-1}} \neq 0)
        = p_A p_{A+1} / (p_A + p_{A+1})^2
        \;.
    \]
    Let $\sigma = \inf \{j > 0 \,:\, d_{k_j} = 0\}$.
    Consider a walk $Z_i$ with the same transition probabilities as $d_{k_i}$ but without stopping,
    such that $Z_i = \ell - d_{k_i}$ for $i \leq \sigma$.
    Note that $\sigma=\{j>0\,:\, Z_j=\ell\}$.
    By the reflection principle,
    \begin{equation*}
    \P(\sigma \leq n) = \P( \max_{i\leq n} Z_i \geq \ell ) = \P(Z_n=\ell) + 2\P(Z_n > \ell) = \P(Z_n=\ell) + \P(|Z_n| > \ell)
    \end{equation*}
    and thus, uniformly in $n,\ell$, $\P(\sigma > n) \leq \P(|Z_n| \leq \ell) \lesssim \ell n^{-1/2}$.
    
    Now let $m,n\geq 0$.
    Using~\eqref{eq:TV_coupling} and that $X_n = Y_n$ for all $n \geq a_{k_\sigma}$, we obtain
    \begin{equation}
    \begin{split}\label{eq:TV_bound}
        \|P^n\delta_0 - P^{n+\ell}\delta_0\|_{\TV}
        & \leq 2 \P(a_{k_\sigma} > n)
        \leq 2 \P ( a_{k_m} \geq n ) + 2\P ( \sigma \geq m )
        \\
        & \lesssim \P ( a_{k_m} \geq n ) + \ell m^{-1/2}
        \;.
    \end{split}
\end{equation}
    It remains to bound $\P ( a_{k_m} \geq n )$ and choose a suitable $m$.
    Let $c_i = \Delta a_{k_i} \bone\{\Delta a_{k_i} \leq B\}$ with $0<B\leq n$.
    Then
    \begin{equation}\label{eq:two_probs}
    \P ( a_{k_m} \geq n ) \leq \P \Bigl( \max_{i\leq m} \Delta a_{k_i} \geq B \Bigr)
    + \P \Bigl( \sum_{i=1}^m c_i \geq n \Bigr) \;.
    \end{equation}
    For the first probability, we have
    \begin{equation}\label{eq:first_prob}
    \P \Bigl( \max_{i\leq m} \Delta a_{k_i} \geq B \Bigr) \leq \sum_{i=1}^m \P(\Delta a_{k_i} \geq B )
    \leq m \rho(B)\;.
    \end{equation}
    For the second probability, we use Hoeffding's inequality.
    For this, using $\P(\Delta a_{k_n}=i) = p_i$, write
    \begin{equation*}
        \E c_1
        = \sum_{i=1}^B i p_i
        \quad \text{and} \quad
        \Var (c_1)
        \leq \E c_1^2
        = \sum_{i=1}^B i^2 p_i
        \;.
    \end{equation*}
    Recall that we assume that $\alpha < 3/2$.
    Then $\Var (c_1) \lesssim B^2 \rho(B)$.
    By~\cite[Theorem~3]{Hoeffding63}, for $m \E c_1 \leq n/2$,
    \begin{equation*}
    \P\Bigl( \sum_{i=1}^m c_i \geq n \Bigr)
    \leq \P\Bigl( \sum_{i=1}^m c_i - m \E c_1 \geq n/2 \Bigr)
    \leq \exp \Bigl[ - \frac{n}{2B} h \Bigl( \frac{2 m \Var (c_1)}{nB} \Bigr) \Bigr]
    \end{equation*}
    where $h(z) = (1+z)\log(1+z^{-1}) - 1$.
    Let $\gamma = \inf_{z\in (0,1)} \frac{h(z)}{\log(z^{-1})}$ and note that $\gamma > 0$. Thus
    \begin{align*}
    \P\Bigl(\sum_{i=1}^m c_i \geq n\Bigr)
    & \leq \exp \Bigl[- \frac{n}{2B} \gamma \log \Bigl( \frac{n B}{2 m \Var (c_1)} \Bigr) \Bigr]
    = \Bigl( \frac{2 m \Var(c_1)}{n B} \Bigr)^{\frac{\gamma n}{2 B}}
    \\
    & \leq K^{\gamma n/B} \Bigl( \frac{m B \rho(B)}{n} \Bigr)^{\frac{\gamma n}{2 B}}
    \;,
    \end{align*}
    where $K$ is uniform in $n,m$ and where we used that, if $\frac{2 m\Var(c_1)}{n B} \geq 1$, then the first bound trivially holds.
    Taking $B = \frac{n}{2 p}$ for $p\geq 1$, we obtain
    \begin{equation*}
    \P \Bigl(\sum_{i=1}^m c_i \geq n \Bigr) \lesssim (m \rho(n))^{\gamma p}
    \end{equation*}
    and therefore, by~\eqref{eq:TV_bound},~\eqref{eq:two_probs}, and~\eqref{eq:first_prob},  for fixed $p \geq \gamma^{-1}$
    and uniformly in $m,n$ such that $m\E c_1 \leq n/2$,
    \begin{equation}
    \begin{split}\label{eq:aksigma}
    \|P^n\delta_0 - P^{n+\ell}\delta_0\|_{\TV}
    & \lesssim \ell m^{-1/2} + m \rho(n) + (m \rho(n))^{\gamma p}
    \\
    & \lesssim \ell m^{-1/2} + m \rho(n)
    \;.
    \end{split}
\end{equation}

    Now we treat the cases $\alpha < 1$, $\alpha = 1$ and $\alpha \in (1, \frac{3}{2})$ separately.

    First suppose that $\alpha<1$. The result is trivial if $\ell^{2/3} \rho(n)^{1/3} > 1$,
    so we assume that $\ell^{2/3} \rho(n)^{1/3} \leq 1$.
    To optimise, we take $m = 1+\floor{\ell^{2/3} \rho(n)^{-2/3}}$,
    which implies $\ell m^{-1/2} \asymp m\rho(n) \asymp \ell^{2/3} \rho(n)^{1/3}$.
    Note that
    \[
        m \E c_1
        \lesssim m B\rho(B)
        \ll m n \rho(n)
        \asymp \ell^{2/3} \rho(n)^{1/3} n
        \leq n
        \;,
    \]
    where the asymptotic ratio can be made arbitrarily large by taking $p$ large.
    Hence with a suitably large $p$ we achieve $m \E c_1 \leq n/2$
    for all sufficiently large $n$.
    Therefore~\eqref{eq:aksigma} holds and we conclude
    $
    \|P^n\delta_0 - P^{n+\ell}\delta_0\|_{\TV} \lesssim \ell^{2/3} \rho(n)^{1/3}
    $.
    
Suppose now that $\alpha=1$. Then $\E c_1 \lesssim M(B) = 1+\int_1^B \rho(t)\,dt$, which is slowly varying in $B$ and increasing.
We take $m=1+\floor{\delta\ell^{2/3} \rho(n)^{-2/3}}$ for some small $\delta>0$.
Now the constraint $m \E c_1\leq n/2$ is satisfied for sufficiently large $n$ if
$\ell \ll \rho(n) M(n)^{-3/2}n^{3/2}$, and the proof proceeds as for $\alpha < 1$.
The case $\ell \leq \rho(n) M(n)^{-3/2}n^{3/2}$ as in the statement of the theorem
follows from a bounded number of applications of the triangle inequality.

Suppose now that $\alpha \in (1,\frac32)$. Then $\E c_1 \lesssim 1$.
To optimise, we take $m$ as a small multiple of $1+\floor{\ell^{2/3} \rho(n)^{-2/3}}$,
    so that $\ell m^{-1/2} \asymp m\rho(n) \asymp \ell^{2/3} \rho(n)^{1/3}$.
    Then, for $n$ sufficiently large, $m\E c_1 \leq \ell^{2/3} \rho(n)^{-2/3} \leq n/2$
    by the assumption that $\ell \leq \rho(n)n^{3/2}$.
    Once again~\eqref{eq:aksigma} holds and we conclude $
    \|P^n\delta_0 - P^{n+\ell}\delta_0\|_{\TV} \lesssim \ell^{2/3} \rho(n)^{1/3}
    $.
The proof of Theorem~\ref{thm:orn} is complete.

\section{Monotonicity}
\label{sec:monotone}

In this section we prove Theorem~\ref{thm:DFR}.
Recall the definition of $u_n$ and $z_n$ in Section~\ref{sec:not}.

\begin{proof}[Proof of Theorem~\ref{thm:DFR}]
    Let $\mu_n = P^n\delta_0$.
    By~\eqref{eq:mu_diff},
    \begin{equation*}
        \|\mu_n - \mu_{n+\ell}\|_{\TV}
        =
        \sum_{m=n+1}^{n+\ell} u_{n+\ell-m} z_{m}
        +
        \sum_{m=0}^n (u_{n-m}-u_{n+\ell-m}) z_m
        \;.
    \end{equation*}
    At the same time, $(\mu_n - \mu_{n + \ell})(\{0, \ldots, n+\ell\}) = 0$, so again
    by~\eqref{eq:mu_diff},
    \begin{equation*}
        \sum_{m=n+1}^{n+\ell} u_{n+\ell-m} z_{m} = \sum_{m=0}^n
        (u_{n-m} - u_{n+\ell-m})z_m\;.
    \end{equation*}
    Hence
    \[
        \|\mu_n-\mu_{n+\ell}\|_{\TV}
        = 2 \sum_{m=n+1}^{n+\ell} u_{n+\ell-m} z_{m}
        \;. \qedhere
    \]
\end{proof}

\subsection{Proof of Lemma~\ref{lem:HF}}
\label{sec:HF}

Recall that $z_0 = 1$ and $z_n = z_{n+1} + p_{n+1}$.
Assume the DFR in this subsection.
Excluding the trivial case $p_1 = 1$ and $p_n = 0$ for $n \geq 2$, assume that
$p_n>0$ for all $n\geq 1$, so that
\[
    \frac{p_n}{z_{n-1}} \geq \frac{p_{n+1}}{z_n}
    \quad \text{for} \quad
    n \geq 1
    \;.
\]

\begin{lemma}
    \label{lem:bn}
    Suppose that $b_k$, $k \geq 0$, is a non-negative sequence with $\frac{b_{k+1}}{b_k} \leq \frac{z_{k+1}}{z_k}$
    for all $k$. Let $c_k = b_{k+1} + p_{k+1} b_0$ for $k \geq 0$. Then
    $\frac{c_{k+1}}{c_k} \leq \frac{z_{k+1}}{z_k}$ for all $k$.
\end{lemma}

\begin{proof}
    Fix $k \geq 0$ and write
    \[
        \frac{c_{k+1}}{c_k}
        = \frac{b_{k+2} + p_{k+2} b_0}{b_{k+1} + p_{k+1} b_0}
        = \lambda\frac{b_{k+2}}{b_{k+1}}+(1-\lambda)\frac{p_{k+2}}{p_{k+1}}
        \;,
    \]
    where $\lambda=\frac{b_{k+1}}{b_{k+1} + p_{k+1} b_0}\in (0,1]$.
    Since $\frac{p_{k+2}}{p_{k+1}} \leq \frac{z_{k+1}}{z_k}$,
    it suffices to consider the case $\frac{b_{k+2}}{b_{k+1}} > \frac{z_{k+1}}{z_k}$.
    Then, since $\lambda$ is monotone decreasing function of $b_0$ (with all other variables fixed),
    so is $\frac{c_{k+1}}{c_k}$.
    In turn, from $\frac{b_{k+1}}{b_0} \leq \frac{z_{k+1}}{z_0}$ we obtain $b_0 \geq \frac{b_{k+1}}{z_{k+1}}$
    and
    \[
        \frac{c_{k+1}}{c_k}
        \leq \frac{
            b_{k+2} + p_{k+2} \frac{b_{k+1}}{z_{k+1}}
        }{
            b_{k+1} + p_{k+1} \frac{b_{k+1}}{z_{k+1}}
        }
        = \frac{
            \frac{b_{k+2}}{b_{k+1}} + \frac{p_{k+2}}{z_{k+1}}
        }{
            1 + \frac{p_{k+1}}{z_{k+1}}
        }
        \leq \frac{
            \frac{z_{k+2}}{z_{k+1}} + \frac{p_{k+2}}{z_{k+1}}
        }{
            1 + \frac{p_{k+1}}{z_{k+1}}
        }
        = \frac{z_{k+1}}{z_k}
        \;.
        \qedhere
    \]
\end{proof}

\begin{cor}\label{cor:t_monontone}
    $u_n \geq u_{n+1}$ for all $n \geq 0$.
\end{cor}

\begin{proof}
    Consider a Markov chain $Y_n$ on non-negative integers with the connection graph~\eqref{eq:RMC}
    but all arrows reversed, starting at $Y_0 = 0$ with transition probabilities
    \[
        \P(Y_{n+1} = k \mid Y_n = k+1)
        = 1 \quad \text{and} \quad
        \P(Y_{n+1} = k \mid Y_n = 0)
        = p_{k+1}
        \quad \text{for} \quad k \geq 0
        \;.
    \]
    Set $b_k^n = \P(Y_n = k)$. Note that $b_k^{n+1} = b_{k+1}^n + p_{k+1} b_0^n$.
    Observe that the renewal probabilities $\P(Y_n = 0 \mid Y_n = 0)$ coincide with those for the Markov chain~\eqref{eq:RMC},
    hence $u_n = b_0^n$. By Lemma~\ref{lem:bn} applied inductively,
    $\frac{b_{k+1}^n}{b_k^n} \leq \frac{z_{k+1}}{z_k}$ for all $n,k$,
    and in particular $\frac{b_1^n}{b_0^n} \leq \frac{z_1}{z_0}$.
    Then
    \[
        \frac{u_{n+1}}{u_n}
        = \frac{b_1^n + p_1 b_0^n}{b_0^n}
        = \frac{b_1^n}{b_0^n} + p_1
        \leq \frac{z_1^n}{z_0^n} + p_1
        = 1
        \;.
        \qedhere
    \]
\end{proof}

\section{Practical rates}
\label{sec:proof_general}
 
In this section we prove Theorem~\ref{thm:general}.
Without loss of generality we assume that $\rho \colon [0,\infty) \to (0,\infty)$ is monotone
and continuous, with $\rho(0)=\rho(1)=1$.
Let us write $\rho(z) = z^{-\alpha} L(z)$ for $z \geq 1$ with a slowly varying function $L$. Define
\[
    K(z) = 1+\int_1^z t^{-1} L(t)^{-2}\, dt
    \;, \quad
    z \geq 1
    \;.
\]
\begin{rmk}
    \label{rmk:K}
    $K(z)$ is strictly increasing in $z$ and $K\geq 1$ and
    \begin{equation}
        \label{eq:KL2}
        K(z) \gtrsim L(z)^{-2}
        \;.
    \end{equation}
    Moreover, by Karamata's theorem from Section~\ref{subsec:RV},
    $K(z)$ is slowly varying.
\end{rmk}
Denote
\[
    \mu_n = P^n \delta_0
    \;.
\]
Throughout this section we assume that~\eqref{eq:p_bound_imp} holds.
Recall $u_n$ from Section~\ref{sec:not}.

\subsection{Proof of (\ref{eq:some_rate})}
\label{sec:proof_some_rate}

We admit for a moment the following lemma, postponing the proof:

\begin{lemma}\label{lem:u_diff}
Suppose~\eqref{eq:p_bound_imp} holds. Then, uniformly in $n\geq 1$,
\begin{equation}\label{eq:un_bound}
u_n\lesssim
\begin{cases}
\rho(n) \quad &\text{if } \alpha \in (0,\frac12)\;,\\
\rho(n) K(n) \quad &\text{if } \alpha=\frac12\;,\\
n^{-1}\rho(n)^{-1} \quad &\text{if } \alpha \in (\frac12,1)\;.
\end{cases}
\end{equation}
Furthermore, uniformly in $\ell\geq 1$ and $n\geq 2\ell$,
\begin{equation}\label{eq:u_diff}
|u_n - u_{n-\ell}| \lesssim
\rho(n)
\times
\begin{cases}
1 \quad &\text{if } \alpha \in (0,\frac12)\;,\\
K (\ell) \quad &\text{if } \alpha = \frac12\;,\\
\ell^{-1} \rho(\ell)^{-2} \quad &\text{if } \alpha \in (\frac12,1)\;.
\end{cases}
\end{equation}
\end{lemma}

\begin{rmk}\label{rem:u_n_bound}
The key result of~\cite{GL62} is that, if $\alpha>1/2$ and $z_n = \rho(n)$,
then $\lim_{n \to \infty} n\rho(n)u_n = \frac{\sin \pi \alpha}{\pi}$.
This is consistent with~\eqref{eq:un_bound}.
\end{rmk}

\begin{lemma}\label{lem:mu_diff}
Suppose that $u_n \leq g(n)$ where $g$ is RV with index in $(-1,0)$. 
Then, uniformly in $n\geq0$, $\ell\geq1$,
\begin{equation*}
\| \mu_{n+\ell} - \mu_n \|_{\TV} \lesssim
\ell g(\ell) \rho(n)
+
\sum_{r=\ell}^{n+\ell} \rho(n+\ell-r) |u_r - u_{r-\ell}|\;.
\end{equation*}
\end{lemma}

\begin{proof}
If $n\leq 4\ell$, then $\rho(n)\gtrsim \rho(\ell)$, so the desired bound follows simply from $\|\mu_n\|_{\TV}=1$.
We thus consider $n> 4\ell$.
By~\eqref{eq:mu_diff},
\begin{equation}\label{eq:mun_diff}
    \| \mu_{n+\ell} - \mu_n \|_{\TV}
    = \sum_{m=n+1}^{n+\ell} u_{n+\ell-m} z_{m}
    + \sum_{m=0}^{n} |u_{n+\ell-m} - u_{n-m}| z_m
    \;.
\end{equation}
By the assumption $u_n \leq g(n)$ and the upper bound in~\eqref{eq:p_bound_imp},
\begin{equation*}
\sum_{m=n+1}^{n+\ell} u_{n+\ell-m} z_{m}
\lesssim
\sum_{m=n+1}^{n+\ell} g(n+\ell-m) \rho(m)
\asymp \ell g(\ell) \rho(n)
\;,
\end{equation*}
where we used $4\ell < n$ and the fact that $g$ is RV with index in $(-1,0)$.

For the second sum in~\eqref{eq:mun_diff}, by the upper bound in~\eqref{eq:p_bound_imp},
\begin{equation*}
\sum_{m=0}^{n} |u_{n+\ell-m} - u_{n-m}| z_m
\leq
\sum_{m=0}^{n} |u_{n+\ell-m} - u_{n-m}| \rho(m)
=
\sum_{r=\ell}^{n+\ell} \rho(n+\ell-r) |u_r - u_{r-\ell}|\;.
\qedhere
\end{equation*}
\end{proof}

\begin{proof}[Proof of~\eqref{eq:some_rate}]
Suppose $\alpha>\frac12$.
We assume that $n\gg \ell$ as otherwise the right-hand side of~\eqref{eq:some_rate} is order $1$.
By~\eqref{eq:un_bound} and Lemma~\ref{lem:mu_diff}, it suffices to bound
\begin{equation*}
\sum_{r=\ell}^{n+\ell} \rho(n+\ell-r) |u_r - u_{r-\ell}|\;.
\end{equation*}
By Lemma~\ref{lem:u_diff}, we can use the following bounds:
\begin{itemize}
    \item $|u_r - u_{r-\ell}| \leq u_r + u_{r-\ell} \lesssim (r-\ell+1)^{-1} \rho(r-\ell)^{-1}$ for $\ell \leq r < 2\ell$,
    \item $|u_r - u_{r-\ell}| \lesssim \rho(r)\ell^{-1}\rho(\ell)^{-2}$ for $2\ell \leq  r$.
\end{itemize} 
The first case contributes, recalling that $n\gg\ell$,
\begin{equation*}
\sum_{\ell \leq r < 2\ell}
\frac{\rho(n+\ell-r)}{(r-\ell+1)\rho(r-\ell)}
\asymp
\frac{\rho(n)}{\rho(\ell)}\;.
\end{equation*}
Note that $\frac{\rho(n)}{\rho(\ell)} \lesssim \frac{n\rho(n)^2}{\ell\rho(\ell)^2}$ since $n\gg\ell$.
The second case contributes,
\begin{equation*}
\sum_{2\ell \leq r \leq n+\ell}
\frac{\rho(n+\ell-r)\rho(r)}{\ell\rho(\ell)}
\lesssim
\frac{n\rho(n)^2 }{\ell\rho(\ell)^{2}}\;.
\end{equation*}
The conclusion follows.
\end{proof}

\begin{rmk}
    Repeating this proof with $\alpha \in (0, 1/2)$ does not yield anything useful: we get
    $\| \mu_{n+\ell} - \mu_n \|_{\TV} \lesssim n \rho(n)^2$ which grows with $n$.
    However, for $\alpha = 1/2$ we obtain
    $\| \mu_{n+\ell} - \mu_n \|_{\TV} \lesssim K(\ell) L(n)^2$, which, depending on $L$, may be small
    for small $\ell$ and large $n$.
\end{rmk}

To finish the proof of~\eqref{eq:some_rate}, it remains to prove Lemma~\ref{lem:u_diff},
for which we require several definitions and lemmas.
Define $P(x) = \sum_{n=1}^\infty p_n e_n(x)$,
where $e_n(x) = e^{\mbi nx}$ and $\mbi=\sqrt{-1}$.
Since $P$ is periodic with period $2\pi$, we will treat it as a function on $[-\pi,\pi)$.

\begin{lemma}
    \label{lem:P1}
    $|P| \leq 1$ with $P(x)=1$ if and only if $x=0$.
\end{lemma}

\begin{proof}
Since $\sum_k p_k=1$ and $p_k \geq 0$ for all $k\geq1$, we have $|P| \leq 1$,
so it remains to show that $P(x)=1$ if and only if $x=0$.

Let $z = x/2\pi\in [-\frac12,\frac12)$. Observe that $P(x) = 1$ if and only if $e^{\mbi m x} = 1$,
or $m z \in \Z$, for each $m$ with $p_m > 0$. By the aperiodicity assumption~\eqref{eq:aperiodic},
there exist $m_1,\ldots,m_k$ such that $\gcd(m_1,\ldots,m_k)=1$ and $p_{m_i}>0$.

Suppose there exists $z \in [-\frac12,\frac12) \setminus \{0\}$ such that $m_i z\in\Z$ for all $i=1,\ldots, k$.
Writing $z= a/b$ with $a,b$ coprime we obtain $z= a/b = c_i/m_i$ for some $c_i\in\Z$, i.e.\ $am_i = bc_i$.
However, if $a,b$ are coprime then $b$ must divide $m_i$ for all $i=1,\ldots, k$, therefore $b=1$ and $z$
is an integer. This is a contradiction.
\end{proof}

For a function $h\colon (0,\infty) \to (0,\infty)$ and $f\colon I \to\C$, where $I\subset \R$, define
\begin{equation*}
|f|_{\CC^h(I)} = \sup_{x\neq y \in I} \frac{|f(x)-f(y)|}{h(|x-y|^{-1})}\;.
\end{equation*}
We will drop $I$ from the subscript whenever $I=[-\pi,\pi)$ and $f$ is $2\pi$-periodic.

\begin{lemma}\label{lem:B_Hol}
Suppose $B = \sum_{m\in\Z} b_me_m$ with
\[
\sum_{|m|\geq n} |b_n| \leq h(n)
\]
where $h$ is non-increasing and regularly varying with index $\beta \in (-1,0]$.
Then $|B|_{\CC^h} \leq C_h$, where $C_h < \infty$ depends only on $h$.
\end{lemma}

\begin{proof}
Fix $\delta \in (0,1)$ and decompose $B=\hat B + \tilde B$ where $\tilde B = \sum_{|k| \geq \delta^{-1}} b_k e_k$.
Then $|\tilde B|_\infty \leq \sum_{|k| \geq \delta^{-1}} |b_k| \leq h(\delta^{-1})$.
On the other hand,
$
|\hat B'|_\infty \leq \sum_{|k| \leq \delta^{-1}} |k b_k|
$.
Since
$
\sum_{2^n \leq |k| \leq 2^{n+1}} |b_k|
\leq h(2^n)
$,
we obtain, uniformly in $\delta\in (0,1)$,
\begin{equation*}
|\hat B'|_\infty \leq \sum_{\substack{0\leq n \leq |\log_2 \delta| \\ 2^n \leq |k| \leq 2^{n+1}}} |kb_k| \leq 2 \sum_{0\leq n \leq |\log_2 \delta|}  2^{n}h(2^n)
\lesssim \delta^{-1} h(\delta^{-1})\;,
\end{equation*}
where we used that $nh(n)$ is regularly varying with index $1+\beta > 0$.
Consequently, for $|x-y| = \delta$,
\[
|B(x)-B(y)| \leq \delta|\hat B'|_\infty + 2|\tilde B|_\infty \lesssim h(\delta^{-1})\;.\qedhere
\]
\end{proof}

\begin{lemma}\label{lem:P_Hol}
$|P|_{\CC^\rho} < \infty$.
\end{lemma}

\begin{proof}
Follows from Lemma~\ref{lem:B_Hol} and the upper bound in~\eqref{eq:p_bound_imp}.
\end{proof}

\begin{lemma}\label{lem:1-P}
    $|1-P(x)| \geq \Re (1-P(x)) \gtrsim \rho(1/|x|)$ uniformly in $x \in [-\pi, \pi)$.
\end{lemma}

\begin{proof}
By Lemma~\ref{lem:P1}, $\Re(1-P(x)) > 0$ if $x \neq 0$, and
it remains to consider the behaviour with $x$ close to $0$.
Since $p_k\geq 0$ and $1-\cos(mx) \gtrsim |mx|^2$ for all $mx\leq \pi$, we obtain
\[
\Re (1-P(x)) = \sum_{m=1}^\infty p_m (1-\cos(mx)) \gtrsim \sum_{m < \pi |x|^{-1}} p_m m^2|x|^2\;.
\]
By our assumptions on $\rho$, there exists $\eps>0$ such that $2\rho(n) \leq c \rho(\eps n)$ for all $n$ sufficiently large
so that $\rho(n)\leq \frac{c}{4}$.
Then, by the two-sided bound~\eqref{eq:p_bound_imp},
\begin{equation*}
\sum_{\eps n \leq m < n} p_m \geq c \rho(\eps n) - \rho(n) \geq \rho(n)\;,
\end{equation*}
and thus
\begin{equation*}
\sum_{\eps n \leq m < n} p_m m^2 \geq \eps^2 n^2 \rho(n)\;.
\end{equation*}
Therefore, for $x$ sufficiently small,
\begin{equation*}
\Re (1-P(x))
\gtrsim |x|^2 \sum_{m < \pi|x|^{-1}} p_m m^2
\gtrsim \rho(1/|x|)\;.\qedhere
\end{equation*}
\end{proof}

For a function $f\in L^1[-\pi,\pi)$, extended periodically to $\R$, and
integer $|n|>1$, note that $\int_{-\pi}^\pi f(x+\pi/|n|) e^{\mbi nx} \,dx = -\int_{-\pi}^\pi f(x) e^{\mbi nx} \,dx$,
and thus
\begin{equation}\label{eq:Fourier_int}
\Big| \int_{-\pi}^\pi f(x)e^{\mbi n x} \,dx\Big|
=
\frac12\Big| \int_{-\pi}^\pi \{ f(x)-f(x+\pi/|n|)\}e^{\mbi n x} \,dx\Big|\;.
\end{equation}
For $\delta>0$, define the $L^1$ modulus of continuity by
$\omega_\delta (f) = \int_{-\pi}^\pi |f(x + \delta) - f(x)| \, dx$.
\begin{lemma}\label{lem:int_f_modulus}
$| \int_{-\pi}^\pi f(x)e^{\mbi n x} \,dx | \leq \frac12\omega_{\pi/|n|}(f)$ for all $|n|\geq 1$.
\end{lemma}
\begin{proof}
Follows immediately from~\eqref{eq:Fourier_int}.
\end{proof}

\begin{lemma}\label{lem:int_over_I}
Let $n\in\Z \setminus \{0\}$ and $\delta = \frac{2\pi}{|n|}$.
Let $I$ be an interval of length $|I|>\delta$.
Then, for $h\colon (0,\infty)\to(0,\infty)$ non-increasing,
\begin{equation*}
\Big|\int_{I} f(x) e^{\mbi n x} \,dx\Big|
\leq 2 \delta |f|_{\infty;I} + h(\delta^{-1}) \int_I |f|_{\CC^h[x-\delta,x+\delta]}\,dx\;.
\end{equation*}
\end{lemma}

\begin{proof}
We subdivide $[-\pi,\pi)$ into intervals $I_k = [k\delta,(k+1)\delta)$.
If $I$ does not contain any $I_k$, then clearly $|I|\leq 2 \delta$ and $|\int_{I} f(x) e^{\mbi n x} \,dx|
\leq 2 \delta |f|_{\infty;I}$.

Suppose now that $I$ contains $I_{k},\ldots,I_{m}$.
The integral of $|f(x)|$ over $I\setminus I_{k}\cup \ldots\cup I_{m}$
is also bounded by $2 \delta |f|_{\infty;I}$, so it suffices to bound
$\sum_{j=k}^m|\int_{I_j} f(x)e^{\mbi nx}\,dx|$.

Note that $\int_{I_j} e^{\mbi nx}\,dx = 0$.
Therefore, since $h$ is non-increasing, $|\int_{I_j} f(x)e^{\mbi nx}\,dx| \leq \delta h(\delta^{-1}) |f|_{\CC^h[y_j-\delta,y_j+\delta]}$
for any point $y_j\in I_j$.
There exist $y_j \in I_j$ such that $\delta |f|_{\CC^h[y_j-\delta,y_j+\delta]} \leq \int_{I_j} |f|_{\CC^h[x-\delta,x+\delta]}\,dx$
and therefore $\sum_{j=k}^m|\int_{I_j} f(x)e^{\mbi nx}\,dx|\leq h(\delta^{-1}) \sum_{j=k}^m \int_{I_j} |f|_{\CC^h[x-\delta,x+\delta]}\,dx\leq h(\delta^{-1}) \int_I |f|_{\CC^h[x-\delta,x+\delta]}\,dx$.
\end{proof}

\begin{proof}[Proof of Lemma~\ref{lem:u_diff}]
Denote $U=1/(1-P)$, so that, by~\eqref{eq:pu}, we have $U = \sum_{n=0}^\infty u_n e_n$ and
$u_n = \frac{1}{2 \pi} \int_{-\pi}^\pi U(x)e^{- \mbi n x} \,dx$.
We first prove~\eqref{eq:un_bound}.
By Lemma~\ref{lem:int_f_modulus} it suffices to prove that, uniformly in $\delta \in (0,1)$,
\begin{equation}\label{eq:U_mod}
\omega_\delta(U)
\lesssim
\rho(1/\delta)\times 
\begin{cases}
1 \quad &\text{if } \alpha \in (0,\frac12)\;,\\
K (1/\delta) \quad &\text{if } \alpha = 1/2\;,\\
\delta \rho(1/\delta)^{-2} \quad &\text{if } \alpha \in (\frac12,1)\;.
\end{cases}
\end{equation}
To this end, by Lemma~\ref{lem:1-P},
\begin{equation}\label{eq:U_x_bound}
|U(x)| = |1-P(x)|^{-1}\lesssim \rho(1/|x|)^{-1}\;,
\end{equation}
and thus, since $\alpha<1$,
\begin{equation}\label{eq:U_int_delta}
\int_{-2\delta}^{2\delta} |U(x+\delta)-U(x)| \,dx \lesssim \delta \rho(1/\delta)^{-1}\;,
\end{equation}
which is smaller than the right-hand side of~\eqref{eq:U_mod} by~\eqref{eq:KL2}.

For non-zero $x \in [-\pi,\pi)$, we claim that
\begin{equation}\label{eq:U_C_alpha}
|U|_{\CC^\rho([\frac12x,\frac32x])} \lesssim \rho(1/|x|)^{-2}\;.
\end{equation}
Indeed, note that $|f\circ g|_{\CC^\rho} \leq |f'|_{\infty}|g|_{\CC^\rho}$.
We apply this with $g\colon [\frac12x,\frac32x] \to \C$ defined by $g=(1-P)\restriction_{[\frac12x,\frac32x]}$ and $f(x)=1/x$, for which $f\circ g = U\restriction_{[\frac12x,\frac32x]}$.
Note that $|g| \gtrsim \rho (1/|x|)$ by Lemma~\ref{lem:1-P} and so $|f'|_{\infty;g[\frac12x,\frac32x]} \lesssim \rho (1/|x|)^{-2}$.
We thus obtain~\eqref{eq:U_C_alpha} since $g\in \CC^\rho$ by Lemma~\ref{lem:P_Hol}.

It follows that, for $|x|>2\delta$,
\begin{equation}\label{eq:U_delta}
|U(x)-U(x+\delta)| \lesssim \rho(1/\delta) \rho(1/|x|)^{-2}\;.
\end{equation}
Therefore
\begin{equation*}
\int_{|x| > 2\delta} |U(x+\delta)-U(x)|\,dx
\lesssim 
\rho (1/\delta) \int_{2\delta}^{\pi} \rho(1/|x|)^{-2} \,dx\;,
\end{equation*}
which, by Karamata's theorem, is bounded from above by the right-hand side of~\eqref{eq:U_mod}.
This completes the proof of~\eqref{eq:U_mod} and thus of~\eqref{eq:un_bound}.

We now prove~\eqref{eq:u_diff}.
Let $\ell\geq 1$.
For $2\ell \leq n \lesssim \ell$,~\eqref{eq:u_diff} follows simply from~\eqref{eq:un_bound} and the triangle inequality.
Suppose henceforth that $n\gg \ell$
and denote $\delta = \pi/(n-\ell) \asymp n^{-1}$.

Denote $F=1-e_{-\ell}$ and $H=FU$.
Observe also that $|F'| \leq \ell$, $F(0)=0$ and $|F|\leq 2$.
Furthermore, for all  $n\geq \ell$, using~\eqref{eq:Fourier_int},
\begin{equation*}
|u_{n-\ell}-u_{n}| = \frac{1}{2\pi}
\Big|\int_{-\pi}^\pi H e_{-n+\ell}\Big|
=
\frac{1}{4\pi}
\Big| \int_{-\pi}^\pi \{ H(x)-H(x+\delta )\}e^{-\mbi (n-\ell) x} \,dx
\Big|\;.
\end{equation*}
By~\eqref{eq:U_x_bound}, $|U(x)| \lesssim \rho(1/|x|)^{-1}$, hence $H(x)\lesssim \ell |x| \rho(1/|x|)^{-1}$,
and thus by triangle inequality
\begin{equation*}
\int_{-2\delta}^{2\delta} |H(x+\delta) - H(x)| \,dx
\lesssim
\int_0^{2\delta} \ell x \rho(1/x)^{-1}\,dx
\lesssim
\frac{\ell \delta^{2}}{\rho(1/\delta)}\;.
\end{equation*}
Remark that $\frac{\ell \delta^{2}}{\rho(1/\delta)}$ is bounded by a multiple of the right-hand side of~\eqref{eq:u_diff}
(for the case $\alpha=\frac12$, this uses $L^{-2}\lesssim K$).

We now consider $|x|>2\delta$ and write
\begin{equation}\label{eq:H}
H(x+\delta)-H(x) = \bigl(F(x+\delta)-F(x)\bigr)U(x) + F(x+\delta)\bigl(U(x+\delta) - U(x)\bigr)\;.
\end{equation}
Using~\eqref{eq:U_delta}, the absolute value of the second term in~\eqref{eq:H} is bounded by a multiple of
$
(\ell |x| \wedge 1) \rho(1/\delta) \rho(1/|x|)^{-2}
$,
which, integrating over $|x|>2\delta$, contributes
\begin{equation*}
\int_{2\delta}^{1/\ell} \ell x \rho(1/\delta) \rho(1/x)^{-2}\,dx
\leq 
\ell\rho(1/\delta) \int_{0}^{1/\ell}  x  \rho(1/x)^{-2}\,dx
\lesssim
\frac{\rho(1/\delta)}{\ell \rho(\ell)^{2}}
\end{equation*}
plus
\begin{equation}\label{eq:2nd_part}
\int_{1/\ell}^{\pi} \rho(1/\delta) \rho(1/x)^{-2}\,dx
\lesssim
\rho(1/\delta)
\times
\begin{cases}
1 \quad &\text{if } \alpha \in (0,\frac12)\;,\\
K (\ell) \quad &\text{if } \alpha = 1/2\;,\\
\ell^{-1} \rho(\ell)^{-2} \quad &\text{if } \alpha \in (\frac12,1)\;.
\end{cases}
\end{equation}
Remark that $\frac{\rho(1/\delta)}{\ell \rho(\ell)^{2}}$ is bounded by a multiple of the right-hand of~\eqref{eq:u_diff},
while~\eqref{eq:2nd_part} is precisely of this order.

For the first term in~\eqref{eq:H}, we divide the integration domain into $|x|\in [2\delta,z]$ and $|x|\in [z,\pi]$ for some $z\geq 2\delta$ to be determined.
For the first domain, we take absolute values and obtain, by~\eqref{eq:U_x_bound},
\begin{equation}\label{eq:F_diff_U_1}
\int_{|x|\in [2\delta,z]} \bigl|\bigl(F(x+\delta)-F(x)\bigr)U(x) \bigr|\,dx
\lesssim
\int_{2\delta}^z \frac{\ell\delta}{\rho(x^{-1})}\,dx \lesssim \frac{\ell\delta z}{ \rho(z^{-1})}\;.
\end{equation}
For the second domain, we use instead
\begin{equation}\label{eq:F_diff_U_2}
\begin{split}
\Big|\int_{|x|>z} \bigl( F(x+\delta) & -F(x) \bigr) U(x) e^{-\mbi (n-\ell) x}\,dx
\Big|
=
\Big|\int_{|x|>z} e_{-\ell}(x) \bigl( e_{-\ell}(\delta) - 1\bigr) U(x) e^{-\mbi (n-\ell) x}\,dx
\Big|
\\
&\leq
2\Big|\int_{|x|>z} U(x) e^{-\mbi n x} \, dx
\Big|
\lesssim
\frac{\delta}{\rho(z^{-1})}
+
\int_{z}^\pi \frac{\rho(\delta^{-1})}{\rho(x^{-1})^2} \, dx
\\
&\lesssim
\rho(\delta^{-1})\times
\begin{cases}
1 \quad &\text{if } \alpha \in (0,\frac12)\;,\\
K (1/z) \quad &\text{if } \alpha = 1/2\;,\\
z\rho(z^{-1})^{-2} \quad &\text{if } \alpha \in (\frac12,1)\;.
\end{cases}
\end{split}
\end{equation}
In the 2nd line above we used Lemma~\ref{lem:int_over_I}
and the fact that $|U(x)| \lesssim \rho(z^{-1})^{-1}$ for $|x|>z$
and $|U|_{\CC^\rho[x-\delta,x+\delta]} \lesssim \rho(|x|^{-1})^{-2}$ by~\eqref{eq:U_C_alpha}.
In the 3rd line we used that $z\geq 2\delta$, that $1/\rho^2$ is RV with index $2\alpha$,
and, for the case $\alpha=\frac12$, the bound~\eqref{eq:KL2}.

For $\alpha<\frac12$ we take $z=2\delta$ (so there is no first domain) and obtain a bound of order $\rho(\delta^{-1})$.
Alternatively, for $\alpha < \frac{1}{2}$ the bound in~\eqref{eq:u_diff} follows directly from~\eqref{eq:un_bound}.

For $\alpha=\frac12$ we choose $z^{-1}=\ell$, so the second domain contributes $\rho(\delta^{-1})K(\ell)$
and the first contributes $\delta/\rho(\ell) \lesssim \rho(\delta^{-1})K(\ell)$,
where the final bound is due to~\eqref{eq:KL2} and $\delta^{-1}\gg \ell$.

For $\alpha > \frac12$ we also choose $z^{-1} = \ell$.
From $\delta^{-1} \asymp n \gtrsim \ell$ we have $\delta^{-1} \rho(\delta^{-1}) \gtrsim \ell \rho(\ell)$
and thus the first domain contributes 
\[
    \frac{\ell\delta z}{ \rho(z^{-1})}
    = \frac{\delta}{\rho(\ell)}
    \lesssim \frac{\rho(\delta^{-1})}{\ell \rho(\ell)^{2}}
    \;.
\]
The second domain contributes
\[
    \frac{\rho(\delta^{-1})}{z^{-1} \rho(z^{-1})^2}
    =
    \frac{\rho(\delta^{-1})}{\ell \rho(\ell)^2}
    \;,
\]
as required.
\end{proof}

\subsection{Proof of~(\ref{eq:mu_reg})}

In addition to~\eqref{eq:p_bound_imp} and~\eqref{eq:aperiodic}, we now assume the regularity bound~\eqref{eq:p_reg}.
We first strengthen the bounds in Lemma~\ref{lem:u_diff} under these assumptions.

\begin{lemma}\label{lem:u_diff_reg}
Uniformly in $n\geq 1$,
\begin{equation}\label{eq:un_bound_imp}
u_n\lesssim
\frac{1}{n\rho(n)}\;.
\end{equation}
Furthermore, uniformly in $\ell\geq 1$ and $n\geq 2\ell$,
\begin{equation}\label{eq:u_diff_imp}
|u_n - u_{n-\ell}|
\lesssim
\begin{cases}
    n^{-1}\rho(n) \rho(\ell)^{-2} &  \alpha < 1/2\;, \\
    n^{-1} \rho(n) \ell \{K(n) - K(\ell)\} &  \alpha = 1/2\;, \\
    n^{-2} \rho(n)^{-1} \ell     &  \alpha > 1/2\;.
\end{cases}
\end{equation}
\end{lemma}

 \begin{rmk}
 If $\alpha>\frac{1}{2}$, then~\eqref{eq:u_diff_imp} for $\ell\geq 1$ follows from
 the case $\ell=1$, i.e.\ $|u_{n}-u_{n-1}| \lesssim n^{-2}\rho(n)^{-1}$, and the triangle inequality.
  Note also that $|u_{n}-u_{n-1}| \lesssim n^{-2}\rho(n)^{-1}$ and $u_n \to 0$ imply the bound~\eqref{eq:un_bound_imp}.
  None of such implications hold in Lemma~\ref{lem:u_diff} or for the cases $\alpha \leq \frac12$.
\end{rmk}

\begin{rmk}
For $n\geq 2\ell$,
\begin{equation}
    \label{eq:KnKell}
    K(n)-K(\ell)
    \gtrsim \int_{\ell}^{2 \ell} \frac{1}{t L(t)^2} \, dt
    + \int_{n / 2}^{n} \frac{1}{t L(t)^2} \, dt
    \gtrsim L(\ell)^{-2} + L(n)^{-2}
    \;.
\end{equation}
Therefore, if $\alpha=\frac12$, then $n^{-1}\rho(n) \rho(\ell)^{-2} + n^{-2} \rho(n)^{-1} \ell \lesssim n^{-1} \rho(n) \ell \{K(n) - K(\ell)\}$.
That is, the bound in~\eqref{eq:u_diff_imp} is compatible with those for $\alpha \neq \frac12$, and is possibly weaker.
\end{rmk}

Using Lemma~\ref{lem:u_diff_reg}, whose proof we postpone, we can complete the proof of Theorem~\ref{thm:general}.

\begin{proof}[Proof of Theorem~\ref{thm:general}]
We already proved~\eqref{eq:some_rate}.
For $n\lesssim \ell$, the bound trivially follows from $\|\mu_k\|_{\TV} \leq 1$,
so we suppose henceforth that $n\gg\ell$.
By~\eqref{eq:un_bound_imp} and Lemma~\ref{lem:mu_diff}, we have
\begin{equation*}
\| \mu_{n+\ell} - \mu_n \|_{\TV} \lesssim
\rho(\ell)^{-1} \rho(n)
+
\sum_{r=\ell}^{n+\ell} \rho(n+\ell-r) |u_r - u_{r-\ell}|\;.
\end{equation*}
Since $\rho(\ell)^{-1} \rho(n)$ is already the desired bound, it remains to bound the sum over $r$.

Suppose first that $\alpha>\frac12$.
By Lemma~\ref{lem:u_diff_reg}, we use the following bounds:
\begin{itemize}
    \item $|u_r - u_{r-\ell}| \lesssim (r-\ell+1)^{-1}\rho(r-\ell)^{-1}$ for $\ell \leq r <2\ell$,
    \item $|u_r - u_{r-\ell}| \lesssim r^{-2} \rho(r)^{-1}\ell$ for $2\ell\leq r$.
\end{itemize}
The first case contributes, since $n\gg\ell$,
\begin{equation*}
\sum_{\ell \leq r < 2\ell} \frac{\rho(n+\ell-r)}{(r-\ell+1)\rho(r-\ell)}
\asymp
\frac{\rho(n)}{\rho(\ell)}
\end{equation*}
where we used that $y^{-1}\rho(y)^{-1}$ is RV with index $-1+\alpha>-1$.
The second case contributes
\begin{equation*}
\sum_{2\ell \leq r \leq n+\ell }
\frac{\ell \rho(n+\ell-r)}{r^{2} \rho(r)}
\asymp
\frac{\ell}{n} + \frac{\rho(n)}{\rho(\ell)}
\asymp
\frac{\rho(n)}{\rho(\ell)}
\end{equation*}
where the first part of the middle term is due to $n/2 \leq r \leq n + \ell$
and the second is due to $\ell \leq r < n/2$.

Suppose now $\alpha < \frac12$.
By Lemma~\ref{lem:u_diff_reg}, we use the following bounds:
\begin{itemize}
    \item $|u_r - u_{r-\ell}| \lesssim (r-\ell+1)^{-1}\rho(r-\ell)^{-1}$ for $\ell \leq r <2\ell$,
    \item $|u_r - u_{r-\ell}| \lesssim r^{-1} \rho(r)\rho(\ell)^{-2}$ for $2\ell\leq r$.
\end{itemize}
The first case contributes $\rho(\ell)^{-1}\rho(n)$ as before.
The second case contributes
\begin{equation*}
\sum_{2\ell \leq r \leq n+\ell }
\frac{\rho(n+\ell-r)\rho(r)}{r \rho(\ell)^2}
\asymp
\frac{\rho(n)^2}{\rho(\ell)^2} + \frac{\rho(n)}{\rho(\ell)}
\asymp
\frac{\rho(n)}{\rho(\ell)}
\end{equation*}
where again the first part of the middle term is due to $n/2 \leq r \leq n + \ell$
and the second is due to $\ell \leq r < n/2$.

Finally, suppose $\alpha=\frac12$.
By Lemma~\ref{lem:u_diff_reg}, we use the following bounds:
\begin{itemize}
    \item $|u_r - u_{r-\ell}| \lesssim (r-\ell+1)^{-1}\rho(r-\ell)^{-1}$ for $\ell \leq r <2\ell$,
    \item $|u_r - u_{r-\ell}| \lesssim r^{-1} \rho(r)\ell \{K(r)-K(\ell)\}$ for $2\ell\leq r$.
\end{itemize}
The first case contributes $\rho(\ell)^{-1}\rho(n)$ as before.
The second case contributes
\begin{equation}\label{eq:final_sum}
\sum_{2\ell \leq r \leq n+\ell }
\frac{\rho(n+\ell-r)\rho(r)\ell \{K(r)-K(\ell)\}}{r}\;.
\end{equation}
We claim that $\eqref{eq:final_sum}\lesssim\frac{\rho(n)}{\rho(\ell)}$.
To prove the claim, let $C = 2^{8}$ and take $y_0$
so that $\frac{L(x)}{L(y)} \in [\frac12,2]$ for all $y > y_0$ and $x\in [y,Cy]$.

Since $n \gg \ell$, we restrict to $n > C \ell$.
We also restrict to $\ell > y_0$, because $\sum_{r= 1}^n \frac{\rho(n-r)\rho(r) K(r)}{r} \lesssim \rho(n)$.

Then, for $t\in [C^i\ell,C^{i+1}\ell]$, we have
$\frac{L(\ell)}{L(t)} = \frac{L(\ell)}{L(C\ell)}\cdots \frac{L(C^i\ell)}{L(t)} \in [2^{-i-1}, 2^{i+1}]$
and therefore
\[
    \int_{C^i \ell}^{C^{i+1}\ell} t^{-1}L(t)^{-2}\,dt
    \leq 4^{i+1} L(\ell)^{-2} \int_{C^i \ell}^{C^{i+1}\ell} t^{-1} \, dt
    = 4^{i+1} L(\ell)^{-2} \log C
    \;.
\]
Then for $C^j\ell \leq r < C^{j+1} \ell$,
\begin{equation}\label{eq:K_diff}
K(r)-K(\ell)
\leq \sum_{i=0}^{j} \int_{C^i \ell}^{C^{i+1}\ell} t^{-1}L(t)^{-2}\,dt
\leq 4^{j+2} L(\ell)^{-2} \log C\;.
\end{equation}

Remark that $\sum_{C^j\ell \leq r < C^{j+1} \ell} r^{-3/2}L(r)
\leq C 2^{j+1} (C^{j}\ell)^{-1/2} L(\ell)$.
Let $k\geq 1$ be the largest integer such that $C^k\ell < n$. Note that, since $\rho$ is monotone,
$\rho(n + \ell - r) \leq \rho(n/2)$ for $r < C^{k-1} \ell$ and thus for $0\leq j \leq k - 2$,
\begin{align*}
    \sum_{C^j \ell \leq r < C^{j+1}\ell }
    \frac{\rho(n+\ell-r)\rho(r)\ell \{K(r)-K(\ell)\}}{r}
    & \leq 4^{j+2} \ell \rho(n/2) L(\ell)^{-2} \log C
    \sum_{C^j \ell \leq r < C^{j+1}\ell }
    \frac{\rho(r)}{r}
    \\
    & \lesssim 8^j C^{-j/2} \frac{\rho(n)}{\rho(\ell)}
    = 2^{-j} \frac{\rho(n)}{\rho(\ell)}
    \;.
\end{align*}
Thus the contribution of these terms to~\eqref{eq:final_sum} is within the desired bound and
it remains to treat $C^{k-1}\ell \leq r \leq n+\ell$.
Our choice of $C$ and $k$ guarantees that $4^k \lesssim \bigl(\frac{n}{\ell}\bigr)^{1/4}$, hence
$4^k \frac{\rho(n)}{\rho(\ell)} \lesssim \frac{n^{-\frac14}L(n)}{\ell^{-\frac14}L(\ell)} \lesssim 1$.
Using this and~\eqref{eq:K_diff},
\begin{align*}
    \sum_{C^{k-1} \ell \leq r \leq n + \ell }
    & \frac{\rho(n+\ell-r)\rho(r)\ell \{K(r)-K(\ell)\}}{r}
    \\
    & \lesssim 4^{k} \ell L(\ell)^{-2} \frac{\rho(n)}{n} \log C
    \sum_{C^{k-1} \ell \leq r \leq n + \ell }
    \rho(n + \ell - r)
    \lesssim 4^k \frac{\rho(n)^2}{\rho(\ell)^2}
    \lesssim \frac{\rho(n)}{\rho(\ell)}
    \;. \qedhere
\end{align*}
\end{proof}

To complete the proof of Theorem~\ref{thm:general}, it remains to prove Lemma~\ref{lem:u_diff_reg}.
To this end, consider first a sequence $a_m\in\C$, $m\in \Z$, such that (think: $a_m=mp_m$)
\begin{equation}\label{eq:a_reg}
\sum_{|m|\geq n} |a_m - a_{m+1}| \lesssim  h(n)
\end{equation}
where $h$ is non-increasing and is RV with index $\beta \in (-1,0]$.

\begin{rmk}
If, in addition, $a_m\to 0$ as $|m|\to\infty$, then~\eqref{eq:a_reg} implies, for $m> 0$, that $
a_m = \sum_{k=m}^\infty a_{k} - a_{k+1} = O(h(m))$,
and similarly for $m<0$.
In general, the two limits $\lim_{m\to\infty} a_m$ and $\lim_{m\to-\infty} a_m$ exist but can be distinct.
\end{rmk}

Define the distribution $A=\sum_{m\in\Z} a_me_m$ where we recall $e_m(x) = e^{\mbi mx}$ for $x\in \R$.

\begin{lemma}\label{lem:blow_up}
The function $A\colon [-\pi,\pi)\setminus\{0\} \to \C$ satisfies
$|A(x)|\lesssim |x|^{-1}h(|x|^{-1})$ for all $x\in [-\pi,\pi)\setminus\{0\}$
and $|A|_{\CC^h(I)} \lesssim |x|^{-1}$ where $I=(x/2,3x/2)$.
\end{lemma}

\begin{proof}
For $x\in [-\pi,\pi)\setminus\{0\}$,
consider a smooth function $\psi\colon \R\to\R$ which is $2\pi$-periodic and equal to $1/(1-e_{-1})$ on $I$.
Note that $|\psi(x)|\lesssim |x|^{-1}$ since $|1-e_{-1}(x)|\gtrsim |x|$.
Let $B=A(1-e_{-1})$. Then, as distributions on $I$,
\[
A = B \psi\;,
\]
where the product is well-defined since $1-e_{-1}$ and $\psi$ are smooth, and
$B=\sum_{m \in \Z} (a_m - a_{m+1}) e_m$.
It follows from~\eqref{eq:a_reg} and Lemma~\ref{lem:B_Hol} that $|B|_{\CC^h}<\infty$.
Furthermore,
\[
B(0)= \sum_{m\in\Z} a_m -a_{m+1} = 0\;,
\]
where the series is absolutely convergent due to~\eqref{eq:a_reg}.
Therefore $|B(x)|\lesssim h(|x|^{-1})$ and thus
\[
|A(x)| = |B(x)\psi(x)| \lesssim |x|^{-1}h(|x|^{-1})\;.
\]
To obtain an upper bound on $|A|_{\CC^h(I)}$, remark that
$|\psi|_{C^1(I)}\lesssim |x|^{-2}$.
Therefore, for all $y,z\in I$,
\[
|\psi(y)-\psi(z)|
\lesssim |y-z||x|^{-2}
\lesssim \frac{h(|y-z|^{-1})|x|^{-1}}{h(|x|^{-1})}
\]
where in the final bound we used $|y-z|^{-1}\gtrsim |x|^{-1}$
and that $h(n)n$ is RV with index $\beta+1>0$.
Therefore
\[
A(z)-A(y) = B(z)(\psi(z)-\psi(y)) + (B(z)-B(y))\psi(y) = O(h(|y-z|^{-1}) |x|^{-1})\;.
\qedhere
\]
\end{proof}

Coming back to the proof of Lemma~\ref{lem:u_diff_reg},
consider as in Section~\ref{sec:proof_some_rate} the function $P= \sum_{m=1}^\infty p_m e_m$.
Since~\eqref{eq:p_reg} is equivalent to $\sum_{m = n}^{2n-1}|p_m - p_{m+1}| \lesssim n^{-1}\rho(n)$,
\[
\sum_{m\geq n} m |p_{m}-p_{m+1}| \lesssim \rho(n)\;.
\]
Combining with the upper bound in~\eqref{eq:p_bound_imp},
it follows that $a_m := \mbi mp_m$ satisfies~\eqref{eq:a_reg} with $h= \rho$.
Therefore, since $P' = \sum_{m} a_m e_m$, we obtain for $x\in [-\pi,\pi)\setminus\{0\}$
by Lemma~\ref{lem:blow_up}
\begin{align}
\label{eq:P_deriv}
|P'(x)| &\lesssim |x|^{-1}\rho(|x|^{-1})\;,
\\
\label{eq:P_Hol}
|P'|_{\CC^\rho([\frac12 x, \frac32 x])} &\lesssim |x|^{-1}\;.
\end{align}

\begin{rmk}
If $\rho(n) = n^{-\alpha}$, then
Lemmas~\ref{lem:P_Hol}-\ref{lem:1-P} and the bounds~\eqref{eq:P_deriv}-\eqref{eq:P_Hol}
imply that $1-P$ behaves like $|x|^\alpha$ in terms of regularity of its derivatives of order up to $1$.
\end{rmk}

\begin{proof}[Proof of Lemma~\ref{lem:u_diff_reg}]
We first prove~\eqref{eq:un_bound_imp}.
As in the proof of Lemma~\ref{lem:u_diff}, define $U=1/(1-P)$, so that $U = \sum_{n=0}^\infty u_n e_n$.
By Lemma~\ref{lem:int_f_modulus} it suffices to prove that, uniformly in $\delta \in (0,1)$,
\begin{equation}\label{eq:U_mod_imp}
\omega_\delta(U)
\lesssim
\frac{\delta}{\rho(\delta^{-1})}\;.
\end{equation}
Recall~\eqref{eq:U_int_delta}:
$\int_{-2\delta}^{2\delta} |U(x+\delta)-U(x)| \,dx \lesssim \delta / \rho(\delta^{-1})$.
By Lemma~\ref{lem:1-P} and~\eqref{eq:P_deriv},
\begin{equation}\label{eq:U'}
|U'(x)| \leq |P'(x)| \, |1-P(x)|^{-2} \lesssim \frac{|x|^{-1}}{\rho(|x|^{-1})}
\;.
\end{equation}
The final expression, as a function of $|x|^{-1}$, is RV with index $1+\alpha > 1$.
Therefore, by Karamata's theorem, 
\begin{equation*}
\int_{|x| > 2\delta} |U(x+\delta)-U(x)|\,dx
\lesssim
\frac{\delta }{\rho(\delta^{-1})}\;,
\end{equation*}
which completes the proof of~\eqref{eq:U_mod_imp} and thus of~\eqref{eq:un_bound_imp}.

We now prove~\eqref{eq:u_diff_imp}.
Let $\ell\geq 1$.
If $2\ell \leq n\lesssim \ell$, then~\eqref{eq:u_diff_imp} follows simply from~\eqref{eq:un_bound_imp}, the triangle inequality,
and the fact that $K(n) - K(\ell) \asymp L(n)^{-2}$ in this case.

Suppose henceforth that $n\gg \ell$.
Following notation from the proof of Lemma~\ref{lem:u_diff}, denote $F=1-e_{-\ell}$, $U=1/(1-P)$ and $H=FU$.
Define $\delta = \pi/(n-\ell) \asymp n^{-1}$.
Then, by~\eqref{eq:Fourier_int},
\begin{equation}\label{eq:u_n_diff_int}
(n-\ell) |u_{n-\ell}-u_{n}| = \frac{1}{2\pi}\Big|\int_{-\pi}^\pi H' e_{-n+\ell}\Big| 
=
\frac{1}{4\pi}
\Big| \int_{-\pi}^\pi \{ H'(x)-H'(x+\delta )\}e^{-\mbi (n-\ell) x} \,dx
\Big|\;.
\end{equation}
To estimate the final integral, we write
    \begin{equation}\label{eq:H'}
        H'(x)
        = \frac{ - \mbi \ell e^{-\mbi \ell x}}{(1-P)(x)}
        + \frac{(1 - e^{-\mbi \ell x})}{(1-P)^2(x)} P'(x)
        = -\mbi A(x) + B(x) P'(x)
        \;.
    \end{equation}
    Note that, by Lemma~\ref{lem:1-P},
    \begin{equation*}
    |A(x)| \lesssim \frac{\ell}{\rho(|x|^{-1})}\;,
    \end{equation*}
    \begin{equation}
|B(x)| \lesssim \frac{\ell|x|\wedge 1}{\rho(|x|^{-1})^2}
\label{eq:B_bound}
\;.
    \end{equation}
    Therefore, using~\eqref{eq:P_deriv},
    \begin{equation*}
    |H'(x)| \lesssim \frac{\ell}{\rho(|x|^{-1})} + \frac{\ell|x|\wedge 1}{|x| \rho(|x|^{-1})} \lesssim \frac{\ell}{\rho(|x|^{-1})}
    \end{equation*}
    and thus, since $\rho$ is RV with index $-\alpha > -1$,
\begin{equation*}
        \int_{-2 \delta}^{2 \delta} |H'(x + \delta) - H'(x)| \, dx
        \lesssim
         \int_{-2 \delta}^{2 \delta} \frac{\ell}{\rho(|x|^{-1})} \, dx
         \lesssim
        \frac{\ell \delta}{\rho(\delta^{-1})}
        \;.
    \end{equation*}
Suppose now that $|x| > 2 \delta$.
To handle the second term in~\eqref{eq:H'},
by~\eqref{eq:P_deriv} and Lemma~\ref{lem:1-P}, 
\begin{equation*}
|B'(x)| \lesssim \frac{\ell}{\rho(|x|^{-1})^2} + \frac{\ell|x|\wedge 1}{|x|\rho(|x|^{-1})^2}
\lesssim  \frac{\ell}{\rho(|x|^{-1})^2}\;.
\end{equation*}
Therefore, by~\eqref{eq:P_deriv}-\eqref{eq:P_Hol} and~\eqref{eq:B_bound},
\begin{equation*}
|B(x+\delta)P'(x+\delta) - B(x)P'(x)| \lesssim \frac{\delta \ell}{|x|\rho(|x|^{-1})} + \frac{(\ell|x|\wedge 1)\rho(\delta^{-1})}{|x|\rho(|x|^{-1})^2} \;.
\end{equation*}
The integral of the first term over $|x|>2\delta$ is $\ell\delta/\rho(\delta^{-1})$ as earlier,
while the integral of the second term over $|x|>2\delta$ is
\begin{equation*}
\int_{1/\ell}^\pi  \frac{\rho(\delta^{-1})}{|x|\rho(|x|^{-1})^2}\, dx
+
\int_{2\delta}^{1/\ell} \frac{\ell\rho(\delta^{-1})}{\rho(|x|^{-1})^2}
\,dx \;.
\end{equation*}
The first integral is proportional to $\rho(\delta^{-1})/\rho(\ell)^{2}$.
The second integral, recalling that $\rho(z) = z^{-\alpha}L(z)$ and $\delta < \frac{1}{4\ell}$, is equal to
\begin{equation}\label{eq:largest}
\ell\rho(\delta^{-1})\int_\ell^{\frac{1}{2\delta}} \rho(y)^{-2} y^{-2} \,dy
\asymp
\begin{cases}
    \rho(\delta^{-1})\rho(\ell)^{-2} &  \alpha < 1/2\;, \\
    \ell \rho(\delta^{-1})\{K(\delta^{-1}) - K(\ell)\}  &  \alpha = 1/2\;, \\
    \ell \delta \rho(\delta^{-1})^{-1}     &  \alpha > 1/2\;.
\end{cases}
\end{equation}
To handle the first term in~\eqref{eq:H'},
we write
\begin{align*}
A(x+\delta)-A(x)
&= \ell e_{-\ell}(x+\delta) \bigl( U(x+\delta)- U(x) \bigr) + \ell \bigl( e_{-\ell}(x+\delta) - e_{-\ell}(x) \bigr) U(x)
\\
&= A_1(x,\delta) + A_2(x,\delta)\;.
\end{align*}
By~\eqref{eq:U'}, $|U'(x)| \lesssim |x|^{-1}\rho(|x|^{-1})^{-1}$,
hence $|U(x+\delta)- U(x)|\lesssim \delta|x|^{-1}\rho(|x|^{-1})^{-1}$, and thus
\begin{equation*}
\int_{|x|>2\delta} |A_1(x,\delta)| \, dx \lesssim \frac{\ell\delta}{\rho(\delta^{-1})}\;.
\end{equation*}
To handle $A_2(x,\delta) = \ell \bigl( F(x)-F(x+\delta) \bigr) U(x)$, note that, for $z\in [2\delta,\pi]$, by~\eqref{eq:F_diff_U_1},
\begin{equation}\label{eq:A_2_1}
\int_{|x|\in [2\delta,z]} |A_2(x,\delta)| \,dx \lesssim \frac{\ell^2 \delta z}{\rho(z^{-1})}\;.
\end{equation}
Furthermore, similar to~\eqref{eq:F_diff_U_2},
\begin{equation}
    \begin{split}\label{eq:A_2_2}
\Big|\int_{|x|>z} A_2(x,\delta) e^{-\mbi(n-\ell)x}\,dx
\Big|
&\leq 
2 \ell \Big|\int_{|x|>z} U(x) e^{-\mbi(n-\ell)x}\,dx
\Big|
\\
&\lesssim
\frac{\ell \delta}{\rho(z^{-1})}
+
\int_z^\pi \frac{\ell \delta}{x\rho(x^{-1})} \,dx \lesssim \frac{\ell \delta}{\rho(z^{-1})}\;,
\end{split}
\end{equation}
where in the 2nd bound we used Lemma~\ref{lem:int_over_I} with $h(y) = y^{-1}$ and~\eqref{eq:U'}.
Taking $z = \ell^{-1} \geq 2\delta$,
both~\eqref{eq:A_2_1} and~\eqref{eq:A_2_2} yield a contribution of order $\ell\delta \rho(\ell)^{-1} \lesssim \ell \delta \rho(\delta^{-1})^{-1}$.

Finally, note that $\ell\delta/\rho(\delta^{-1})$ and $\rho(\delta^{-1})/\rho(\ell)^{2}$, for any $\alpha\in(0,1)$,
are bounded by a multiple of~\eqref{eq:largest} (for $\alpha=\frac12$ we use
$K(\delta^{-1})-K(\ell) \gtrsim L(\delta^{-1})^{-2} + L(\ell)^{-2}$ since $\delta^{-1} \gg \ell$, as in~\eqref{eq:KnKell}).
Therefore~\eqref{eq:u_n_diff_int} is bounded by the right-hand side of~\eqref{eq:largest}, concluding the proof.
\end{proof}

\appendix

\section{Memory loss for positive recurrent Markov chains}
\label{sec:Lind}

Here we justify~\eqref{eq:best}. We use notation from Section~\ref{sec:not},
in particular the sequences $u_n = \P(X_n=0 \mid X_0=0)$ 
and $z_n = \sum_{k > n} p_k$.
We start by recalling the following special case of~\cite[Theorem~2.2]{Mitov_Omey_14_renewal}:

\begin{thm}\label{thm:un_pos_rec}
Suppose that $\alpha>1$. If $z_n \lesssim n^{-\alpha}$, then $|u_n-u_{n+1}|\lesssim n^{-\alpha}$.
\end{thm}

To prove~\eqref{eq:best}, consider $\alpha>1$.
It suffices to consider $\ell=1$ and apply the triangle inequality.
By~\eqref{eq:mu_diff} and Theorem~\ref{thm:un_pos_rec},
\begin{align*}
    \|P^{n+1} \delta_0 - P^n \delta_0\|_{\TV}
    & = z_{n+1} + \sum_{m=0}^n |u_{n+1-m}-u_{n-m}|z_m\
    \\
    & \lesssim n^{-\alpha} + \sum_{m=0}^n (n-m+1)^{-\alpha} (m+1)^{-\alpha}
    \lesssim n^{-\alpha}
    \;,
\end{align*}
as desired.

\begin{rmk}
The above proof of~\eqref{eq:best} relies on Theorem~\ref{thm:un_pos_rec}, the proof of which in~\cite{Mitov_Omey_14_renewal} is based on a generalised Wiener lemma.
There are other ways to prove~\eqref{eq:best}, for example we were able to adapt~\cite[Theorem~4.2,~p27]{Lindvall_02_coupling},
which essentially contains~\eqref{eq:best} under the strong moment assumption $\sum_{n\geq 1}n^{-\alpha}p_n<\infty$.
\end{rmk}

\subsection*{Acknowledgements}
I.C.\ acknowledges support from the DFG CRC/TRR 388 ``Rough Analysis, Stochastic Dynamics and Related Fields'' through a Mercator Fellowship held at TU Berlin
and from the EPSRC via the New Investigator Award EP/X015688/1.
A.K.\ was partially supported by EPSRC grant EP/V053493/1,
and is grateful to Wael Bahsoun for support, and to
Ian Melbourne and Ronald Zweim\"uller for helpful discussions and advice.
The authors are grateful to the Institute of Advances Studies at Loughborough University
for their hospitality and support of this collaboration.

\bibliographystyle{Martin}
\bibliography{refs}{}

\end{document}